\documentclass[a4paper,10pt]{article}

\usepackage[english]{babel} 

\usepackage[A4]{vmargin} 
\usepackage{multicol} 
\usepackage{graphicx} 
\usepackage{color, colortbl} 
\usepackage{lscape} 

\usepackage{bbm}
\usepackage{dsfont}
\usepackage{amsfonts}
\usepackage{amssymb}
\usepackage{mathrsfs}
\usepackage{amsthm}
\usepackage{amsmath}
\DeclareMathAlphabet{\mathpzc}{OT1}{pzc}{m}{it}
\usepackage{stmaryrd}
\usepackage{amsthm}
\usepackage{hyperref}
\hypersetup{                    
    colorlinks=true,                
    breaklinks=true,                
    urlcolor= blue,                 
    linkcolor= red,                
    citecolor= blue                
    }
\usepackage{verbatim} 
\usepackage{enumerate} 
\usepackage{eurosym} 
\usepackage{algorithm,algorithmic} 
\usepackage{pstricks,pst-node,pst-tree,pstricks-add}

\usepackage[toc,page]{appendix}

\usepackage{makeidx}

\newtheorem{theorem}{Theorem}[section]
\newtheorem{proposition}[theorem]{Proposition}
\newtheorem{corollary}[theorem]{Corollary}
\theoremstyle{definition}
\newtheorem{definition}[theorem]{Definition}

\theoremstyle{plain}
\newtheorem{remark}{Remark}[section]
\theoremstyle{plain}

\theoremstyle{notation}

\numberwithin{equation}{subsection} 

\setmarginsrb{3cm}{2cm}{3cm}{2cm}{2cm}{0cm}{0cm}{1cm}

\newcommand\pare[1]{\left(#1\right)}
\newcommand\va[1]{\left|#1\right|}

\newcommand\croch[1]{\left[#1\right]}

\newcommand{\dsp}{\displaystyle}
\newcommand{\mbf}{\mathbf}

\newcommand{\R}{\mbf{R}}

\newcommand{\T}{\mbf{T}}

\title{High-Field Limit from a Stochastic BGK Model to a Scalar Conservation Law with Stochastic Forcing}
\author{Nathalie Ayi}
\date{}

\makeindex

\begin{document}
\maketitle

\begin{abstract}
\noindent We study the derivation of a scalar conservation law with stochastic forcing starting from a stochastic BGK model with a high-field scaling. We prove the convergence to a new kinetic formulation where appears a modified Maxwellian. We deduce from it the existence of a weak solution to the scalar conservation law with stochastic forcing. We establish that this solution satisfies some Krushkov-like entropy relations. 
\end{abstract}

\noindent \textbf{Keywords.} Stochastic conservation laws, high field limit, hydrodynamical limit, BGK model, kinetic formulation, modified Maxwellian, Krushkov-like entropies.

\section{Introduction}

In this paper, we are interested in the derivation of a solution to a scalar conservation law with stochastic forcing of the following form: 
\begin{equation}
\label{stochasticlaw}
du + div_x(B(x,u)) dt = C(x,u) dW_t, ~~~~~~t \in (0,T), ~x \in \T^N,
\end{equation}
from a stochastic BGK like model, 
\begin{equation}
\label{highfieldBGK}
 \dsp{dF^\varepsilon + div_x ( a(x,\xi)F^\varepsilon) dt + \frac{\Lambda(x)}{\varepsilon} \partial_\xi F^\varepsilon dt = \frac{\mathbf{1}_{u_\varepsilon> \xi} - F^\varepsilon}{\varepsilon} dt - \partial_\xi F^\varepsilon \Phi dW + \frac{1}{2} \partial_\xi (G^2 \partial_\xi F^{\varepsilon}) dt}, 
\end{equation}
when $\varepsilon$ goes to $0$. \\

\noindent This type of problem belongs to the class of problems of  hydrodynamical limits. The historical issue in that field goes back to the kinetic theory, first introduced by Maxwell and Boltzmann in order to model rarefied gas. Indeed, adopting a statistical point of view, a gas can be described by its  density of particles $f_\varepsilon$, a function of time, position and velocity, satisfying at  the mesoscopic level, the Boltzmann equation
\begin{equation}
\partial_t f_\varepsilon(t,x,\xi) + v. \nabla_x f_\varepsilon(t,x,\xi) = \frac{Q(f_\varepsilon,f_\varepsilon)}{\varepsilon}, ~~~~~ t \in \R_t^+, x  \in \R^d_x, \xi \in \R^n_\xi,
\end{equation}
where $Q(f_\varepsilon,f_\varepsilon)$ is a collision operator. Formally in that case, we obtain the following hydrodynamical limit: $f_\varepsilon$ converges to a Maxwellian of parameter $\rho(t,x)$, $u(t,x)$ and $T(t,x)$ such that those three functions satisfy the Euler system. In other words, starting from the Boltzmann equation at the mesoscopic level, we expect to obtain the Euler system at the macroscopic level. However, this derivation is actually still an open question and motivated the introduction of the BGK model in \cite{BGK}, tough partial results have been established as for example in the case of the incompressible Euler system for ``well prepared initial data'' of the asymptotic equations or solution to the scaled Boltzmann equation with additional non uniform a priori estimates (see \cite{limitehydroboltz} for a more complete state of the art). The idea of the BGK model  is to replace the collision kernel by an easier term which keeps some properties associated to the Boltzmann equation. Those models have been then generalized by  Perthame and Tadmor \cite{PerTad} to derive at the macroscopic level some scalar conservation laws. More precisely, the BGK model that they investigate in \cite{PerTad} is the following: 
\begin{equation}
\label{BGKdeter}
(\partial_t + a(\xi). \partial_x) f_\varepsilon(t,x,\xi) = \frac{1}{\varepsilon} \croch{\chi_{u_\varepsilon(t,x)}(\xi)-f_\varepsilon(t,x,\xi)}, ~~~~~~~(t,x,\xi) \in \R_t^+ \times \R_x^d \times \R_\xi
\end{equation}
with  $\dsp{u_\varepsilon(t,x) = \int_\R f_\varepsilon(t,x,\xi) d\xi}$ and where 
$
\dsp{\chi_u(\xi) }= \left\{
\begin{array}{rl}
\dsp{\text{sgn} ~u } &  \text{if} ~(u-\xi) \xi \geq 0,\\
0&  \text{if} ~(u-\xi) \xi < 0.
\end{array} \right. 
$\\
At the macroscopic level, they derive the following multidimensional law 
\begin{equation}
\partial_t u(t,x) + \sum_{i=1}^d \partial_{x_i} A_i(u(t,x))=0, ~~~~~~~(t,x) \in \R_t^+ \times \R_x^d,
\end{equation}
with  $A_i \in \mathcal{C}^1(\R)$ and $a_i = A_i'$.
Two main approaches are actually possible to get this scalar conservation law from  \eqref{BGKdeter}. The initial one in \cite{PerTad} is based on a use of BV compactness arguments in the multidimensional case and compensated compactness arguments in the one dimensional case, while the second one introduced later in \cite{LionsPer} relies on what is called a kinetic formulation. It consists in obtaining at the macroscopic level a kinetic equation involving a density-like function whose velocity distribution is the equilibrium density. \\

\noindent Lately, the study of scalar conservation laws starting from BGK models at the mesoscopic level have been pursued in various context:  with boundary conditions in \cite{Nouri} and \cite{BerthelinBouchut}, with a discontinuous in the space variable flux in \cite{Discont}, with a high-field scaling in \cite{BerthPoup} or again in a stochastic context in \cite{Hof}. In this paper, we are interested in the last two contexts mentioned. First of all, let us say a word about the stochastic one. Recently, the study of some conservation laws with stochastic forcing has been a subject of growing interest (see \cite{Bauz}, \cite{DebVov},\cite{Weinan}, \cite{Feng}, \cite{Hold}, \cite{Kim} or again \cite{BerthelinVovellesystem} for the case of a system). Indeed, the introduction of such terms can be justified in order to translate numerical and empirical uncertainties. Moreover, it often offers the possibility to weaken assumptions and still get results. The first result on an hydrodynamical limit starting from a BGK model in a stochastic context is due to Hofmanova \cite{Hof}. The idea in this paper is to use the notion of stochastic kinetic formulation developed by Debussche and Vovelle in \cite{DebVov}. Certainly, it seems quite natural to adopt this approach knowing that we can not get any pathwise $L^\infty$ a priori estimates due to the presence of the white noise term.\\

\noindent What we intend to do here is to extend the above result to a stochastic BGK model containing a force term with a high-field scaling. The deterministic version of this result has been established by Berthelin, Poupaud and Mauser in \cite{BerthPoup}. However, due to the presence of the white noise terms, the techniques adopted through the paper will be the ones developed by Debussche and Vovelle \cite{DebVov} and Hofmanov\'a \cite{Hof} by proving the convergence to a kinetic formulation associated to \eqref{stochasticlaw}. The scaling adopted being different, some new difficulties are introduced and some additional assumptions similar to the ones of the deterministic version will be adopted. Moreover, a modified Maxwellian will be obtained in the kinetic formulation leading by its properties to get some Krushkov-like inequalities. \\

\noindent The organization of the paper is the following: in section \ref{context} we will present the context and state the main result. Section \ref{BGKmodel} will be dedicated to the study of the stochastic BGK model \eqref{highfieldBGK}. We will state its existence and prove its convergence to a kinetic formulation. Finally, in section \ref{conservationlaw} we will conclude to the existence of a weak solution to \eqref{stochasticlaw} which satisfies some Krushkov-like entropy relations.

\section{Settings and Main Results}
\label{context}

\noindent In the following, we will denote by $\mathcal{C}^{k,\mu}$ with $ k \in \mathbf{N}$ and $\mu >0$ the set of  $k$ times differentiable $\mu$-H\"older functions. Let $(\Omega, \mathcal{F}, (\mathcal{F}_t)_{t \geq 0}, \mathbb{P})$ be a stochastic basis with a complete, right-continuous filtration. We can assume without loss of generality that the $\sigma-$algebra $\mathcal{F}$ is countably generated and $(\mathcal{F}_t)_{t \geq 0}$ is the completed filtration generated by the Wiener process and the initial condition. We denote by $\mathcal{P}$ the predictable $\sigma$-algebra on $\Omega \times [0,T]$ associated to $(\mathcal{F}_t)_{t \geq 0 }$ and by $\mathcal{P}_s$ the predictable $\sigma$-algebra on $\Omega \times [s,T]$ associated to $(\mathcal{F}_t)_{t \geq s }$. We write $L_{\mathcal{P}_s}^\infty(\Omega \times [s,T] \times \T^N \times \R)$ to denote 
\begin{equation*}
L^\infty( \Omega \times [s,T] \times \T^N \times \R, ~ \mathcal{P}_s \otimes \mathcal{B}(\T^N) \otimes \mathcal{B}(\R), ~ d \mathbb{P} \otimes dt \otimes dx \otimes d \xi ). 
\end{equation*}
and idem for $L_{\mathcal{P}}^\infty(\Omega \times [s,T] \times \T^N \times \R)$ 
\noindent The setting of our paper has some similarities with the ones of 
Debussche and Vovelle \cite{DebVov} and Hofmanov\'a \cite{Hof}. Indeed, we will assume that we work on a finite-time interval $[0,T]$, $T>0$ with periodic boundary conditions, $x$ belonging to $\mathbf{T}^N$ where $\mathbf{T}^N$ is the  $N$-dimensional torus. We will consider a  fonction 
 \begin{equation}
 a=(a_1, \dots, a_N): \T^N \times \R \to \R^N
 \end{equation}
 of class $\mathcal{C}^{3,\mu}$ for some $\mu>0
 $ such that its  satisfies for all $x,\xi$ ,
\begin{equation}
\label{positivitediv}
 0 \leq div_x  (a(x,\xi))
\end{equation} 
 and for all $i=1, \dots, N$, for all $x \in \T^N$, 
 \begin{equation}
 \label{hypo_a}
 \int_{\R} |\xi ~ a_i(x,\xi) |d \xi < + \infty.
 \end{equation}
We assume that W is a d-dimensional $(\mathcal{F}_t)$-Wiener process, defined as follows
 \begin{equation}
 \label{wiener}
 W(t) = \sum_{i=1}^d \beta_k(t) e_k
 \end{equation}
 where $(\beta_k)_{k=1}^d$ are mutually independent Brownian processes, $(e_k)_{k=1}^d$ an orthonormal basis of $H$ a finite dimensional Hilbert space. \\
 \noindent For each $u \in \R$, $\Phi(u): H \to  L^2(\T^N)$ is defined by $\Phi(u)e_k = g_k(u)$ where $g_k(.,u)$ is a regular function on $\T^N$. More precisely, the functions  $g_1, \dots, g_d: \T^N \times \R \to \R$ are of class $\mathcal{C}^{4,\mu}$  with linear growth and bounded derivatives of all orders. In that context, the following estimate holds true 
 \begin{equation}
 \label{inegG}
 G^2(x,\xi) = \sum_{k=1}^d |g_k(x,\xi)|^2 \leq C(1+|\xi|^2), ~~~~  \forall x \in \T^N, \xi \in \R .
 \end{equation}
Moreover, we assume that 
\begin{equation}
\label{gk0}
g_k(x,0)=0, ~~~ \forall x \in \T^N, ~k=1, \dots, d,
\end{equation}
and for all $x \in \T^N$ and for all $k$,
\begin{equation}
\label{hypo_g}
\int_\R | \xi ~ \partial_\xi g_k(x,\xi)| d \xi < + \infty.
\end{equation}
In addition, we assume that 
\begin{equation}
\label{negativitelambda}
 \Lambda: \T^N \to \R ~\text{is a nonpositive function of class} ~\mathcal{C}^{4,\mu}.
\end{equation} 

\noindent Regarding the initial data, we suppose that $u_0 \in L^p(\Omega; L^p(\T^N))$ for all $p \in [1, \infty)$ and we will consider $F_0=\mathbf{1}_{u_0 > \xi}$.\\

\noindent Finally, quite similarly to the deterministic case with a high-field scaling \cite{BerthPoup}, we will assume that 
\begin{equation}
\label{hypo_feps}
(\int_\mathbf{R} |\xi ~ f_\varepsilon(\omega,t,x,\xi)| d \xi)  ~\text{is bounded in}~ L^\infty(\Omega \times [0,T] \times \T^N), 
\end{equation}
where $f_\varepsilon := F_\varepsilon + \mathbf{1}_{0>\xi}$.\\

\noindent We finally can state our main result: 
\begin{theorem}
Under  assumptions \eqref{positivitediv}, \eqref{hypo_a}, \eqref{wiener} \eqref{inegG}, \eqref{gk0}, \eqref{hypo_g}, \eqref{negativitelambda} and \eqref{hypo_feps}, for any $\varepsilon>0$ there exists a weak solution to the stochastic BGK model with a high field scaling \eqref{highfieldBGK} denoted by $F_\varepsilon$ with initial condition $F_0 = \mathbf{1}_{u_0> \xi}$. Moreover, $F_\varepsilon$ belongs to $L_{\mathcal{P}}^\infty(\Omega \times [0,T] \times \T^N \times \R)$ and converges weak-$*$ to $M_u$ a modified Maxwellian associated to $u$ where $u$ is a weak solution to the conservation law with stochastic forcing \eqref{stochasticlaw} with 
\begin{equation}
B(x,u) = \int_{-\infty}^u \int_0^{+\infty} a(x, \xi+v\Lambda(x)) e^{-v} dv d\xi
\end{equation}
and 
\begin{equation}
C(x,u) = \int_{-\infty}^u \int_0^{+\infty} \partial_\xi \Phi(x,\xi+v\Lambda(x)) e^{-v} dv d\xi.
\end{equation}
Denoting   $u_\varepsilon(t,x) := \int_\R f_\varepsilon(t,x,\xi) d\xi$, the sequence of local densities $(u_\varepsilon)_{\varepsilon>0}$ converges to the weak solution $u$ in $L^p(\Omega \times [0,T] \times \T^N)$ for all $p \in [1, + \infty)$. In addition, $u$ satisfy the Krushkov-like inequalities \eqref{Krushkovineg}. 
\end{theorem}

\section{Study of the stochastic BGK model with a high field scaling}
\label{BGKmodel}
\subsection{The stochastic kinetic model}
\label{stoch_kin_mod}
We are interested in the following stochastic BGK model:
\begin{equation}
\label{BGKstoc}
  \left\{ 
\begin{array}{lcr}
 \dsp{dF^\varepsilon + div_x ( a(x,\xi)F^\varepsilon) dt + \frac{\Lambda(x)}{\varepsilon} \partial_\xi F^\varepsilon dt = \frac{\mathbf{1}_{u_\varepsilon> \xi} - F^\varepsilon}{\varepsilon} dt - \partial_\xi F^\varepsilon \Phi dW + \frac{1}{2} \partial_\xi (G^2 \partial_\xi F^{\varepsilon}) dt,}  \\ 
F_\varepsilon(0,x,\xi)= \mathbf{1}_{u_0(x)> \xi}(\xi).
\end{array}
\right. 
\end{equation}

\noindent The definition of a weak solution of this system is the following:
\begin{definition}
Let $\varepsilon >0$, $F_\varepsilon \in L_{\mathcal{P}}^\infty(\Omega \times [0,T] \times \T^N \times \R)$ is called a weak solution of \eqref{BGKstoc} if for any test function $\phi \in \mathcal{C}_c^\infty(\T^N \times \R)$, we have a.e. $t \in [0,T]$, $\mathbb{P}$-a.s.
\begin{multline}
<F_\varepsilon(t), \phi> = <F_0, \phi> + \int_0^t <F_\varepsilon(s), a.\nabla \phi> ds + \frac{1}{\varepsilon} \int_0^t <F_\varepsilon(s), \Lambda(x)\partial_\xi \phi > ds \\
+ \frac{1}{\varepsilon} \int_0^t  <\mathbf{1}_{u_\varepsilon(s) > \xi} - F_\varepsilon(s), \phi(s)> ds + \sum_{k=1}^d \int_0^t <F_\varepsilon(s),\partial_\xi(g_k \phi)> d \beta_k(s)\\
+ \frac{1}{2} \int_0^t <F_\varepsilon(s),\partial_\xi(G^2 \partial_\xi \phi)> ds.
\end{multline}
\end{definition}

\noindent Let us state here several results. For  $\Theta$  a smooth function with compact support satisfying $0 \leq \Theta \leq 1$ and 
\begin{equation}
  \Theta(\xi) := \left\{ 
\begin{array}{lcr}
1 ~$if$~ |\xi| \leq 1/2,  \\ 
0 ~$if$~ |\xi| \geq 1,
\end{array}
\right. 
\end{equation}
we denote $\Theta_R(\xi):= \Theta(\frac{\xi}{R})$, $g_k^R(x,\xi) := g_k(x,\xi) \Theta_R(\xi)$ for $k=1, \dots, d$ and $a^R(x,\xi):= a(x, \xi)  \Theta_R(\xi)$. The coefficients $\Phi^R$ and $G^{R,2}$ are defined similarly as $\Phi$ and $G^{2}$ replacing $g_k$ by $g_k^R$. We introduce the intermediate problem,
\begin{equation}
\label{aux}
  \left\{ 
\begin{array}{lcr}
 \dsp{dX^\varepsilon + div_x ( a(x,\xi)X^\varepsilon) dt + \frac{\Lambda(x)}{\varepsilon} \partial_\xi X^\varepsilon dt = - \partial_\xi X^\varepsilon \Phi dW + \frac{1}{2} \partial_\xi (G^{2} \partial_\xi X^{\varepsilon}) dt,}  \\ 
X^\varepsilon(s)= X_0,
\end{array}
\right. 
\end{equation}
and the truncated problem associated
\begin{equation}
\label{trunc}
  \left\{ 
\begin{array}{lcr}
 \dsp{dX^\varepsilon + div_x ( a^R(x,\xi)X^\varepsilon) dt + \frac{\Lambda(x)}{\varepsilon} \partial_\xi X^\varepsilon dt = - \partial_\xi X^\varepsilon \Phi^R dW + \frac{1}{2} \partial_\xi (G^{R,2} \partial_\xi X^{\varepsilon}) dt,}  \\ 
X^\varepsilon(s)= X_0.
\end{array}
\right. 
\end{equation}

\begin{proposition}
\label{existencenontrunc}
If  $X_0 $ is a $\mathcal{F}_s \otimes \mathcal{B}(\T^N) \otimes  \mathcal{B}(\R)$- measurable initial data belonging to $L^\infty(\Omega \times \T^N \times \R)$, then there exists a weak solution $X^\varepsilon \in L_{\mathcal{P}_s}^\infty(\Omega \times [s,T] \times \T^N \times \R)$ to \eqref{aux}. Moreover, it is represented by 
\begin{equation}
\mathcal{S}^\varepsilon(t,s)(\omega,x,\xi) := \lim_{R \to + \infty} [\mathcal{S}^{\varepsilon,R}(t,s) X_0](\omega,x,\xi), ~~~~~ 0 \leq s\leq t \leq T,
\end{equation}
with $\mathcal{S}_{\varepsilon,R}$ a solution operator of \eqref{trunc}. 
\end{proposition}

\noindent From above, we obtain the following proposition:
\begin{proposition}
For any $\varepsilon>0$, there exists a weak solution of the stochastic BGK model with high field scaling \eqref{highfieldBGK} denoted by $F_\varepsilon$. Moreover, $F_\varepsilon$ is represented by 
\begin{equation}
\label{RepresentationFeps}
F_\varepsilon(t) = e^{-t/\varepsilon} \mathcal{S}^\varepsilon(t,0) \mathbf{1}_{u_0>\xi} + \frac{1}{\varepsilon} \int_0^t e^{-\frac{t-s}{\varepsilon}} \mathcal{S}^\varepsilon(t,s) \mathbf{1}_{u_\varepsilon(s)>\xi} ds.
\end{equation}
\end{proposition}

\noindent The proof of those two results is quite similar to the ones in \cite{Hof} and mainly relies on a use of the stochastic characteristics method developed by Kunita in \cite{Kun}. Nevertheless, there is some additional difficulties due to the presence of the term $\frac{\Lambda(x)}{\varepsilon} \partial_\xi F_\varepsilon$ (which introduce a dependence on $\varepsilon$ of the solution operator) and the dependence of $a$ in $x$. Then, for the sake of completness, we give a quite detailed sketch of the proof of the existence in Appendix \ref{wellposed}.

\subsection{Convergence of the stochastic kinetic model}

\noindent From now on, let us denote by $C$ a constant which does not depend on any parameter and may change from a line to another. In this section, our purpose is to study the limit of the stochastic kinetic model \eqref{highfieldBGK} as $\varepsilon$ goes to $0$ in the following weak formulation satisfied by $F_\varepsilon$: 

\begin{multline}
\label{formulationfaibleHFSBGK}
\int_0^T < F_\varepsilon(t), \partial_t \varphi(t)> dt + <F_0, \varphi(0)> + \int_0^T <F_\varepsilon(t), a. \nabla \varphi(t)> dt \\
= - \frac{1}{\varepsilon} \int_0^T <\mathbf{1}_{u_\varepsilon>\xi} - (F_\varepsilon(t)+ \Lambda(x) \partial_\xi F_\varepsilon(t)), \varphi(t)> dt  + \int_0^T < \partial_\xi F_\varepsilon(t) \Phi dW_t, \varphi(t)> \\
+ \frac{1}{2} \int_0^T <G^2 \partial_\xi F_\varepsilon(t), \partial_\xi \varphi(t)>dt.
\end{multline}
 for $\varphi \in \mathcal{C}_c^\infty([0,T) \times \T^N \times \R)$.
 
\begin{proposition}
\label{convergenceFeps}
Up to a subsequence, $(F_\varepsilon)_{\varepsilon >0}$ converges weak-$*$ to $F$ with $F$ satisfying the following:  for any test function $\varphi \in \mathcal{C}_c([0,T) \times \T^N \times \R)$, 
\begin{multline}
\label{formulationcinetique}
\int_0^T < F(t), \partial_t \varphi(t)> dt + <F_0, \varphi(0)> + \int_0^T <F(t), a. \nabla \varphi(t)> dt \\
= m( \partial_\xi \varphi) + \int_0^T < \partial_\xi F(t) \Phi dW_t, \varphi(t)> + \frac{1}{2} \int_0^T <G^2 \partial_\xi F(t), \partial_\xi \varphi(t)>dt
\end{multline}
with $m$ a random nonnegative bounded Borel measure on $[0,T] \times \T^N \times \R$ and where $m(\partial_\xi \varphi)$ denotes 
\begin{equation}
m( \partial_\xi \varphi) := \int_{\T^N \times [0,T] \times \R} \partial_\xi \varphi dm(x,t,\xi).
\end{equation}
\end{proposition}

\begin{proof}
From the representation formula of $F_\varepsilon$ \eqref{RepresentationFeps}, we deduce that the set of solutions $\{F_\varepsilon; \varepsilon \in (0,1)\}$ is bounded in $L_{\mathcal{P}}^\infty(\Omega \times [0,T] \times \T^N \times \R)$. Using the Banach-Alaoglu theorem, we know that there exists $F$ in $L_{\mathcal{P}}^\infty(\Omega \times [0,T] \times \T^N \times \R)$ such that, up to subsequences, as $\varepsilon$ goes to $0$,
\begin{equation*}
F_\varepsilon  \mathop{\rightharpoonup}\limits^{w-*} F~ \text{in} ~  L_{\mathcal{P}}^\infty(\Omega \times [0,T] \times \T^N \times \R).
\end{equation*}
Then, almost surely, we have as $\varepsilon$ goes to $0$,
\begin{equation*}
\int_0^T <F_\varepsilon(t), \partial_t \varphi(t)> dt \longrightarrow \int_0^T <F(t), \partial_t \varphi(t)> dt,
\end{equation*}
\begin{equation*}
\int_0^T <F_\varepsilon(t), a.\nabla \varphi(t)> dt \longrightarrow \int_0^T <F(t), a.\nabla \varphi(t)> dt,
\end{equation*}
\begin{equation*}
\frac{1}{2} \int_0^T <G^2 \partial_\xi F_\varepsilon(t), \partial_\xi \varphi(t)> dt \longrightarrow \frac{1}{2} \int_0^T <G^2 \partial_\xi F(t), \partial_\xi \varphi(t)> dt.
\end{equation*}
Moreover, regarding the stochastic term, a use of dominated convergence theorem for stochastic integrals allows us to conclude that almost surely, 
\begin{equation*}
 \int_0^T < \partial_\xi F_\varepsilon(t) \Phi dW_t, \varphi(t)>  \to \int_0^T   < \partial_\xi F(t) \Phi dW_t, \varphi(t)> .
\end{equation*}
Indeed, $<F_\varepsilon, \partial_\xi(g_k \varphi)> \longrightarrow <F, \partial_\xi(g_k \varphi)>$ a.e. $(\omega,t) \in \Omega \times [0,T]$. In addition, by assumption on $g_k$ and since $F_\varepsilon$ is bounded,
\begin{equation*}
|<F_\varepsilon, \partial_\xi(g_k \varphi)>| \leq C.
\end{equation*}
It remains to deal with the first term of the right-hand side of \eqref{formulationfaibleHFSBGK}. Let us define 
\begin{equation}
\label{mesure}
m_\varepsilon(t,x,\xi) := \frac{1}{\varepsilon} \int_{-\infty}^\xi (\mathbf{1}_{u_\varepsilon(t,x)>\zeta}(\zeta)- F_\varepsilon(t,x,\zeta)) d\zeta + (-\Lambda(x) F_\varepsilon(t,x,\xi)).
\end{equation}

\noindent We can prove similarly as in \cite{Hof}
that the first term of the right-hand side of  \eqref{mesure} is a random nonnegative measure over $[0,T] \times T^N \times \R$ and since by assumption \eqref{negativitelambda}, $\Lambda(x) \leq 0$ for all $x$, so does the second term. Indeed, we recall that a positive distribution is a positive Radon measure. \\

\noindent  Moreover, $f_\varepsilon$ satisfies
\begin{multline}
df^\varepsilon + div_x ( a(x,\xi)f^\varepsilon) dt + div_x ( a(x,\xi) \mathbf{1}_{0>\xi}) dt  \\
= \partial_\xi m_\varepsilon - \partial_\xi f^\varepsilon \Phi dW_t  - \partial_\xi \mathbf{1}_{0>\xi} \Phi dW_t + \frac{1}{2} \partial_\xi (G^2 \partial_\xi F^{\varepsilon}) dt.
\end{multline}
since $f_\varepsilon = F_\varepsilon - \mathbf{1}_{0>\xi}$.\\

\noindent We consider the following test functions 
\begin{equation*}
\varphi(t,x,\xi) = \Psi(t,x) \tilde{\varphi}(\xi)
\end{equation*}
with $\Psi \in \mathcal{C}_c([0,T) \times \T^N)$, $\tilde{\varphi} \in \mathcal{C}_c(\R)$. We denote 
\begin{equation*}
H(t,x,\xi) := \int_{-\infty}^\xi \varphi(t,x,\zeta) d\zeta.
\end{equation*}
Thus, using the fact that $\partial_\xi \mathbf{1}_{0>\xi} = - \delta_0$ and assumption \eqref{gk0}, we have 
\begin{equation}
\begin{array}{rcl}
\dsp{\mathbb{E} \int_0^T <m_\varepsilon, \varphi> dt} & = & \dsp{- \mathbb{E} \int_0^T <\partial_\xi m_\varepsilon, H(t)> dt}\\
& = &  \dsp{\mathbb{E} \int_0^T <f_\varepsilon(t), \partial_t H(t)> dt +  \mathbb{E} <f_0, \varphi(0)> }\\
& ~& + ~\dsp{\mathbb{E} \int_0^T <f_\varepsilon(t), a \cdot \nabla H(t)> dt + \mathbb{E} \int_0^T <\mathbf{1}_{0>\xi}, a \cdot \nabla H(t)> dt }\\
&~ & - ~\dsp{\mathbb{E} \int_0^T <\partial_\xi f_\varepsilon(t) \Phi dW_t,  H(t)>  - \frac{1}{2} \mathbb{E} \int_0^T <G^2 \partial_\xi F_\varepsilon(t),  \partial_\xi  H(t)> dt.}
\end{array}
\end{equation}
 At this point, our aim is to bound $\mathbb{E} \int_0^T <m_\varepsilon, \varphi> dt$ for all $\varphi$ in $\mathcal{C}_c([0,T) \times \T^N \times \R)$ independently of $\varepsilon$ in order to conclude thanks to properties of Radon measure on locally compact spaces. Let us first prove that the term $\mathbb{E} \int_0^T <\partial_\xi f_\varepsilon \Phi dW_t,  H(t)>$ disappears. We recall that a martingale has zero expected value. Therefore, the idea is to use the stochastic Fubini theorem (see \cite[Theorem 4.18]{Daprato}) to interchange the integrals with respect to $x, \xi$ and the stochastic one. In order to apply this theorem, we must prove that 
 \begin{equation*}
 \int_{\T^N} \int_\R \pare{\mathbb{E}  \int_0^T \left|   f_\varepsilon(t) \partial_\xi (g_k H(t))\right|^2 dt }^{1/2} d\xi dx < + \infty 
 \end{equation*}
 for each $k=1,2, \dots, d$. Indeed, we have
 
 \begin{equation}
 \begin{array}{l}
 \dsp{\int_{\T^N} \int_\R \pare{\mathbb{E}  \int_0^T \left|   f_\varepsilon(t) \partial_\xi (g_k H(t)) \right|^2 dt }^{1/2} d\xi dx}\\
  \dsp{= \int_{\T^N} \int_\R \pare{\mathbb{E}  \int_0^T \left|   f_\varepsilon(t) \partial_\xi g_k H(t) + f_\varepsilon(t) g_k \varphi \right|^2 dt }^{1/2} d\xi dx}\\
  \leq \dsp{ \int_{\T^N} \int_\R \croch{\pare{\mathbb{E}  \int_0^T \left|   f_\varepsilon(t) \partial_\xi g_k H(t)\right|^2 dt }^{1/2} + \pare{\mathbb{E}  \int_0^T \left|   f_\varepsilon(t) g_k \varphi\right|^2 dt }^{1/2}} d\xi dx}\\
  \leq C \dsp{\pare{\int_{\T^N} \int_\R  |\partial_\xi g_k \xi| d\xi dx + \int_{\T^N} \int_\R  (1+ |\xi|^2)^{1/2} \tilde{\varphi} d\xi dx}}\\
  \dsp{< + \infty}
 \end{array}
 \end{equation}
using the fact that $f_\varepsilon$ is bounded and assumption \eqref{hypo_g}. Then, we can interchange the integrals. It remains to prove that $\int_0^T   \partial_\xi f_\varepsilon g_k H(t)  dW_t$ is a well defined martingale. For each $k=1, \dots, d$, we have
\begin{multline}
 \mathbb{E} \int_0^T \left|   f_\varepsilon(t) \partial_\xi(g_k H(t))\right|^2 dt\\
 \leq C  \pare{\mathbb{E} \int_0^T \left|  f_\varepsilon(t) \partial_\xi  g_k H(t)\right|^2 dt + \mathbb{E} \int_0^T \left|  f_\varepsilon(t)  g_k \varphi \right|^2 dt}\\
 < + \infty
\end{multline}
by the same arguments as previously. So, finally we have
\begin{equation}
\begin{array}{rcl}
\dsp{\left| \mathbb{E} \int_0^T <m_\varepsilon, \varphi> dt\right|} & \leq &  \dsp{\mathbb{E} \int_0^T <f_\varepsilon(t), |\partial_t \Psi(t) \xi|> dt +  \left|\mathbb{E} <f_0, \varphi(0)> \right| }\\
& ~& + ~\dsp{\mathbb{E} \int_0^T <f_\varepsilon(t), |a. \nabla \Psi(t) \xi|> dt +  \mathbb{E} \int_0^T <\mathbf{1}_{0>\xi}, |a. \nabla \Psi(t) \xi|> dt  }\\
&~ & + ~\dsp{ \frac{1}{2} \left| \mathbb{E} \int_0^T <G^2 \partial_\xi F_\varepsilon(t),  \varphi(t)> dt \right|.}
\end{array}
\end{equation}
Then by assumptions \eqref{hypo_a} and \eqref{hypo_g}, we deduce that $\dsp{\mathbb{E} \int_0^T <m_\varepsilon, \varphi> dt}$ is bounded  independently of $\varepsilon$. By density, the results hold true for all $\varphi \in \mathcal{C}_c^\infty([0,T) \times \T^N \times \R)$ and similarly for all $\varphi \in \mathcal{C}_c([0,T) \times \T^N \times \R)$. Then, by properties of Radon measures on locally compact spaces, we can  conclude that for almost every $\omega \in \Omega$, there exists a nonnegative measure $m(\omega)$ such that, up to a subsequence, almost surely, 
\begin{equation}
\int_0^T <m_\varepsilon, \varphi(t)> dt \longrightarrow \int_0^T <m, \varphi(t)> dt
\end{equation}
for any test function $\varphi \in \mathcal{C}_c^\infty([0,T) \times \T^N \times \R)$. \\

\noindent Finally, passing to the limit in \eqref{formulationfaibleHFSBGK} leads to the statement. 
\end{proof}

\begin{remark}
The purpose of assumptions \eqref{hypo_a} and \eqref{hypo_g} can actually be found in the above proof to bound $\mathbb{E} \int_0^T <m_\varepsilon, \varphi> dt$. Indeed, due to the dependence on $\varepsilon$ of the stochastic characteristic system because of the presence of the term $\frac{\Lambda(x)}{\varepsilon} \partial_\xi F_\varepsilon$, the techniques developed in \cite{Hof} can not be adopted here. 
\end{remark}

\subsection{A kinetic formulation}

As mentioned previously, a kinetic formulation is a kinetic equation at the macroscopic level satisfy by a density-like function. At this point, we already have obtained \eqref{formulationcinetique}, the kinetic equation satisfy by $F$ at the macroscopic level. Then, it remains to study the behavior of its velocity distribution. \\

\noindent In order to do so, let us recall here some notions that will be needed to get the kinetic formulation associated to \eqref{stochasticlaw}.

\begin{definition}[Young measure]
\label{youngmeasure}
Let $(X,\lambda)$ be a finite measure space. Let $\mathcal{P}_1(\R)$ denote the set of probability measures on $\R$. We say that a map $\nu: X \to \mathcal{P}_1(\R)$ is a Young measure on $X$ if, for all $\phi \in \mathcal{C}_b(\R)$, the map $z \mapsto \nu_z(\phi)$ from $X$ to $\R$ is measurable.
\end{definition}
\noindent We have the following compactness result:

\begin{theorem}[Compactness of Young measure]
\label{compactnessyoungmeasure}
Let $(X, \lambda)$ be a finite measure space such that $L^1(X)$ is separable. Let $(\nu^n)$ be a sequence of Young measures on $X$ satisfying for some $p \geq 1$, 
\begin{equation*}
\sup_n \int_X \int_\R |\xi|^p d\nu_z^n(\xi)d\lambda(z) < + \infty.
\end{equation*}
Then, there exists a Young measure $\nu$ on $X$ and a subsequence still denoted $(\nu^n)$ such that, for all $h \in L^1(X)$, for all $\phi \in \mathcal{C}_b(\R)$,
\begin{equation*}
\lim_{n \to + \infty} \int_X h(z) \int_\R \phi(\xi) d\nu_z^n(\xi)d\lambda(z) =   \int_X h(z) \int_\R \phi(\xi) d\nu_z(\xi)d\lambda(z).
\end{equation*}
\end{theorem}

\noindent The proof, quite classical, can be found in \cite{DebVov}. \\

\noindent The purpose of this subsection is then actually to prove the following result:
\begin{proposition}
\label{FegalMu}
There exists $u \in L^1(\Omega \times [0,T] \times \T^N)$ such that for all $t \in [0,T]$ for almost every $(x,\xi,\omega)$, $F=M_u$ where we denote by $M_u$ the modified Maxwellian solution to the equation
\begin{equation}
M_u + \Lambda(x) \partial_\xi M_u = \mathbf{1}_{u> \xi}.
\end{equation}
\end{proposition}

\begin{proof}
Multiplying \eqref{BGKstoc} by $\varepsilon$, then  we have
\begin{equation}
\label{convdistrib}
\mathbf{1}_{u_\varepsilon> \xi} - F_\varepsilon - \Lambda(x) \partial_\xi F_\varepsilon \longrightarrow 0
\end{equation}
in the sense of distributions over $(0,T) \times \T^N \times \R$ almost surely. So, we deduce that 
\begin{equation*}
\partial_\xi \croch{\mathbf{1}_{u_\varepsilon> \xi} - F_\varepsilon - \Lambda(x) \partial_\xi F_\varepsilon} \longrightarrow 0
\end{equation*}
or again
\begin{equation}
\label{convergencesuitesmesuredirac}
- \delta_{u_\varepsilon= \xi} - \partial_\xi F_\varepsilon - \Lambda(x) \partial_\xi^2 F_\varepsilon \longrightarrow 0
\end{equation}
in the sense of distributions over $(0,T) \times \T^N \times \R$ almost surely.\\

\noindent We set $\nu_{t,x}^\varepsilon := \delta_{u_\varepsilon(t,x) = \xi}$ which by Definition \ref{youngmeasure} is a Young measure. Then, by assumption \eqref{hypo_feps}, we have
\begin{equation}
\label{bornenueps}
\begin{array}{rcl}
\dsp{\sup_{\varepsilon>0} \sup_{t \in [0,T]} \int_{\T^N} \int_\R |\xi| d \nu_{t,x}^\varepsilon(\xi) dx }& =  &\dsp{ \sup_{\varepsilon>0} \sup_{t \in [0,T]} \int_{\T^N} |u_\varepsilon(t,x)| dx }\\
& \leq & \dsp{\sup_{\varepsilon>0} \sup_{t \in [0,T]} \int_{\T^N}\int_\R |f_\varepsilon(t,x,\xi)| d\xi dx < + \infty.} 
\end{array}
\end{equation}
Using Proposition \ref{compactnessyoungmeasure}, we know that there exists a Young measure $\nu_{t,x}$ such that $\nu_{t,x}^\varepsilon \to \nu_{t,x}$ (up to a subsequence). Then, we deduce from \eqref{convergencesuitesmesuredirac} and Proposition \ref{convergenceFeps} that
\begin{equation}
\label{mesurenu}
\partial_\xi F + \Lambda(x) \partial_\xi^2 F = - \nu.
\end{equation}

\noindent Let us prove then that $ F + \Lambda(x) \partial_\xi F \in \{0,1\}$. Let us construct the following mollifier on $[0,T]$: 
\begin{equation}
\label{theta}
\theta(t) :=  \left\{ \begin{array}{l}
\tilde{C} \exp \pare{\frac{1}{t^2-1}} ~ \text{if} ~ t \leq 1,\\
0~  \text{else}
\end{array} \right.
\end{equation}
with $\tilde{C}$ the constant such that $\dsp{\int_0^T \theta(t) dt=1}$. We denote for $0 \leq \delta \leq 1$, 
\begin{equation}
\label{thetadelta}
\theta_\delta(t) :=  \frac{1}{\delta} \theta \pare{\frac{t}{\delta}}. 
\end{equation}
Then we have
\begin{equation}
\label{inegtheta}
|\theta_\delta(t)| \leq \frac{C}{\delta} ~\text{and} ~ |\theta_\delta'(t)| \leq \frac{C}{\delta^2}, 
\end{equation}
with $C$ a constant which does not depend on $\delta$. We set for $s \in [0,T]$,
\begin{equation*}
\varphi_1(t,x,\xi) = \theta_\delta(s-t) \Psi_1(x,\xi) ~\text{and} ~ \varphi_2(t,y,\zeta) = \theta_\delta(s-t) \Psi_2(y,\zeta) 
\end{equation*}
with $\Psi_1 \in \mathcal{C}_c^\infty(\T^N_x \times \R_\xi)$, $\Psi_2 \in \mathcal{C}_c^\infty(\T^N_y \times \R_\zeta)$. \\

\noindent By \eqref{formulationfaibleHFSBGK}, we have that 
\begin{equation*}
\begin{array}{l}
\dsp{\int_0^T < F_\varepsilon(t)+ \Lambda(x) \partial_\xi F_\varepsilon(t), \varphi_1(t)> dt  \times \int_0^T <1- (F_\varepsilon(t)+ \Lambda(x) \partial_\xi F_\varepsilon(t)), \varphi_2(t)> dt} \\
\dsp{= \left[  \int_0^T <\mathbf{1}_{u_\varepsilon(t)>\xi}, \varphi_1(t)> dt \right. }\\
\dsp{~~~~~~~~ + \varepsilon \left( \int_0^T < F_\varepsilon(t), \partial_t \varphi_1(t)> dt + <F_0, \varphi_1(0)> + \int_0^T <F_\varepsilon(t), a. \nabla \varphi_1(t)> dt \right.} \\
\dsp{~~~~~~~~~~~~~~~~~~ - \left. \left. \int_0^T < \partial_\xi F_\varepsilon(t) \Phi dW_t, \varphi_1(t)>  - \frac{1}{2} \int_0^T <G^2 \partial_\xi F_\varepsilon(t), \partial_\xi  \varphi_1(t)>dt \right) \right]}\\
\dsp{ \times \left[  \int_0^T <1-\mathbf{1}_{u_\varepsilon(t)>\zeta}, \varphi_2(t)> dt \right. }\\
\dsp{~~~~~~~~ - \varepsilon \left( \int_0^T < F_\varepsilon(t), \partial_t \varphi_2(t)> dt + <F_0, \varphi_2(0)> + \int_0^T <F_\varepsilon(t), a. \nabla \varphi_2(t)> dt \right.} \\
\dsp{~~~~~~~~~~~~~~~~~~ - \left. \left. \int_0^T < \partial_\xi F_\varepsilon(t) \Phi dW_t, \varphi_2(t)>  - \frac{1}{2} \int_0^T <G^2 \partial_\xi F_\varepsilon(t), \partial_\xi  \varphi_2(t)>dt \right) \right]}\\
\end{array}
\end{equation*}
\begin{equation*}
\begin{array}{l}
\dsp{\leq \pare{\int_0^T <\mathbf{1}_{u_\varepsilon(t)>\xi}, \varphi_1(t)> dt } \times \pare{\int_0^T <1-\mathbf{1}_{u_\varepsilon(t)>\zeta}, \varphi_2(t)> dt} }\\
\dsp{~~~~~+ \varepsilon \pare{1+\frac{C}{\delta} \int_0^T <1, |\Psi_2|> dt} \pare{ \left|\int_0^T  <F(t), \partial_t \varphi_1(t)> dt \right|+ |\Xi_1 + r_1(\varepsilon)|}} \\
\dsp{~~~~~+ \varepsilon ~\frac{C}{\delta} \int_0^T <1, |\Psi_1|> dt \pare{\left| \int_0^T  <F(t), \partial_t \varphi_2(t)> dt \right| + |\Xi_2 + r_2(\varepsilon)|} }\\
\dsp{~~~~~+ \varepsilon^2 \prod_{i=1}^2 \pare{\left| \int_0^T  <F(t), \partial_t \varphi_i(t)> dt\right| + |\Xi_i + r_i(\varepsilon)|}},
\end{array}
\end{equation*}
where for $i=1,2$, 
\begin{multline*}
\Xi_i := <F_0, \varphi_i(0)> + \int_0^T <F(t), a. \nabla \varphi_i(t)> dt \\
- \int_0^T < \partial_\xi F(t) \Phi dW_t, \varphi_i(t)>  - \frac{1}{2} \int_0^T <G^2 \partial_\xi F(t), \partial_\xi  \varphi_i(t)>dt,
\end{multline*}
 $r_i(\varepsilon)$ exists and is a function such that $r_i(\varepsilon)  \mathop{\longrightarrow}\limits_{\varepsilon \to 0} 0$ by Proposition \ref{convergenceFeps}. Then, we deduce from \eqref{inegtheta} that 
\begin{equation*}
\begin{array}{l}
\dsp{\int_0^T < F_\varepsilon(t)+ \Lambda(x) \partial_\xi F_\varepsilon(t), \varphi_1(t)> dt  \times \int_0^T <1- (F_\varepsilon(t)+ \Lambda(x) \partial_\xi F_\varepsilon(t)), \varphi_2(t)> dt} \\
\dsp{\leq \pare{\int_0^T <\mathbf{1}_{u_\varepsilon(t)>\xi}, \varphi_1(t)> dt } \times \pare{\int_0^T <1-\mathbf{1}_{u_\varepsilon(t)>\zeta}, \varphi_2(t)> dt} }\\
\dsp{~~~~~+ \varepsilon \pare{1+\frac{C}{\delta} \int_0^T <1, |\Psi_2|> dt} \pare{ \frac{C}{\delta^2} \int_0^T  <F(t), | \Psi_1(t)|> dt + \Xi_1 + r_1(\varepsilon)}} \\
\dsp{~~~~~+ \varepsilon ~\frac{C}{\delta} \int_0^T <1, |\Psi_1|> dt \pare{\frac{C}{\delta^2}  \int_0^T  <F(t), | \Psi_2(t)|> dt + \Xi_2 + r_2(\varepsilon)} }\\
\dsp{~~~~~+ \varepsilon^2 \prod_{i=1}^2 \frac{C}{\delta^2}  \pare{\int_0^T  <F(t), |\Psi_i(t)|> dt + |\Xi_i + r_i(\varepsilon)|}}.
\end{array}
\end{equation*}
Moreover, since 
\begin{equation}
\label{convregultemps1}
\begin{array}{l}
\dsp{\mathbb{E} \croch{\int_0^T < F_\varepsilon(t)+ \Lambda(x) \partial_\xi F_\varepsilon(t), \varphi_1(t)> dt   \times \int_0^T <1- (F_\varepsilon(t)+ \Lambda(x) \partial_\xi F_\varepsilon(t)), \varphi_2(t)> dt}}\\
 \dsp{~~~~~~~~~ \mathop{\longrightarrow}\limits_{\delta \to 0} \mathbb{E}  \croch{< F_\varepsilon(s)+ \Lambda(x) \partial_\xi F_\varepsilon(s), \Psi_1(s)> <1 - ( F_\varepsilon(s)+ \Lambda(x) \partial_\xi F_\varepsilon(s)), \Psi_2(s)>} }
\end{array}
\end{equation}
and 
\begin{equation}
\label{convregultemps2}
\begin{array}{l}
\dsp{\mathbb{E} \croch{\int_0^T < \mathbf{1}_{u_\varepsilon(t)>\xi}, \varphi_1(t)> dt   \times \int_0^T <1- \mathbf{1}_{u_\varepsilon(t)>\zeta}, \varphi_2(t)> dt}}\\
 \dsp{~~~~~~~~~ \mathop{\longrightarrow}\limits_{\delta \to 0} \mathbb{E}  \croch{< \mathbf{1}_{u_\varepsilon(s)>\xi}, \Psi_1(s)> <1 - \mathbf{1}_{u_\varepsilon(s)>\zeta}, \Psi_2(s)>}}, 
\end{array}
\end{equation}
we deduce that 
\begin{multline*}
\dsp{\mathbb{E}  \croch{< F_\varepsilon(s)+ \Lambda(x) \partial_\xi F_\varepsilon(s), \Psi_1(s)> <1 - ( F_\varepsilon(s)+ \Lambda(x) \partial_\xi F_\varepsilon(s)), \Psi_2(s)>+ r_3(\delta)}} \\
\\
\dsp{\leq \mathbb{E}  \left[ < \mathbf{1}_{u_\varepsilon(s)>\xi}, \Psi_1(s)> <1 - \mathbf{1}_{u_\varepsilon(s)>\zeta}, \Psi_2(s)> + r_4(\delta) \right.}\\
\dsp{ \left. + C \frac{\varepsilon}{\delta^3} (1+ |r_1(\varepsilon|)+|r_2(\varepsilon))| + C \frac{\varepsilon^2}{\delta^4} \prod_{i=1}^2 (1+ |r_i(\varepsilon)|) \right]}
\end{multline*}
where by \eqref{convregultemps1} and \eqref{convregultemps2}, $r_3(\delta)$ and $r_4(\delta)$ exist and are functions such that $r_3(\delta), r_4(\delta) \mathop{\longrightarrow}\limits_{\delta \to 0} 0.$ We denote $\alpha(x,\xi,y,\zeta)= \psi_1(x, \xi) \Psi_2(y,\zeta)$ and $\ll .,. \gg$ the duality distribution over $\T^N_x \times \R_\xi \times \T^N_y \times \R_\zeta.$ Then we can rewrite the above inequality as follows
\begin{multline*}
\dsp{\mathbb{E}  \ll  \pare{F_\varepsilon(s)+ \Lambda(x) \partial_\xi F_\varepsilon(s)} \pare{1 - ( F_\varepsilon(s)+ \Lambda(x) \partial_\xi F_\varepsilon(s))}, \alpha \gg+ r_3(\delta)} \\
\\
\dsp{\leq \mathbb{E}  \ll \mathbf{1}_{u_\varepsilon(s)>\xi} \pare{1- \mathbf{1}_{u_\varepsilon(s)>\zeta}}, \alpha \gg + r_4(\delta)}\\
\dsp{+ C \frac{\varepsilon}{\delta^3} (1+ |r_1(\varepsilon)|+|r_2(\varepsilon))| + C \frac{\varepsilon^2}{\delta^4} \prod_{i=1}^2 (1+ |r_i(\varepsilon)|)}.
\end{multline*} 
By a density argument, it remains true for any test function $\alpha \in \mathcal{C}_c^\infty(\T^N_x \times \R_\xi \times \T^N_y \times \R_\zeta)$. We consider $\varrho_{\eta_1}$ and $\rho_{\eta_2}$ some mollifiers on respectively $\T^N$ and $\R$. For any $R>0$, we take 
\begin{equation*}
\alpha(x,\xi,y,\zeta) = \varrho_{\eta_1}(x-y) \rho_{\eta_2}(\xi - \zeta) \Theta_R(|x|) \Theta_R(\xi).
\end{equation*}
Passing to the limit $\eta_1, \eta_2 \longrightarrow 0$, and $\varepsilon \longrightarrow 0$ with $\delta = \varepsilon^{1/4}$, we finally get 
\begin{equation*}
\mathbb{E} <F(s)+\Lambda(x) \partial_\xi F(s), \pare{1 - (F(s)+\Lambda(x) \partial_\xi F(s))} \Theta_R(|x|) \Theta_R(\xi)> \leq 0.
\end{equation*}
Moreover, by \eqref{convdistrib} there exists a distribution $r_6$ such that 
\begin{equation*}
F_\varepsilon + \Lambda(x) \partial_\xi F_\varepsilon = \mathbf{1}_{u_\varepsilon > \xi} + r_6(\varepsilon)
\end{equation*}
with $r_6(\varepsilon) \mathop{\longrightarrow}\limits_{\varepsilon \to 0} 0$. \\
Then 
\begin{equation*}
r_6(\varepsilon) \leq F_\varepsilon + \Lambda(x) \partial_\xi F_\varepsilon \leq 1 + r_6(\varepsilon)
\end{equation*}
and  passing to the limit, we get 
\begin{equation*}
0 \leq F+ \Lambda(x) \partial_\xi F \leq 1.
\end{equation*}
Finally, 
\begin{equation*}
\mathbb{E} <F(s)+\Lambda(x) \partial_\xi F(s), \pare{1 - (F(s)+\Lambda(x) \partial_\xi F(s))} \Theta_R(|x|) \Theta_R(\xi)> = 0.
\end{equation*}
Thus for almost every $x,\xi,\omega$, $\forall s \in [0,T]$, we deduce from above that 
\begin{equation*}
F(s,x,\xi,\omega) + \Lambda(x) \partial_\xi F(s,x,\xi,\omega) \in \{0,1\}. 
\end{equation*}
Since we have
\begin{equation*}
\begin{array}{rcl}
F+\Lambda(x) \partial_\xi F & = & \dsp{- \int_0^\xi - \partial_\xi(F+\Lambda(x) \partial_\xi F)d\zeta}\\
 & = &\dsp{ - \int_0^\xi  \nu d\zeta}
\end{array}
\end{equation*}
with $\nu$ a Young measure, there exists $u(x,t,\omega) \in \R$ such that 
\begin{equation}
F+\Lambda(x) \partial_\xi F = \mathbf{1}_{u(t) > \xi}.
\end{equation}
Moreover, we can deduce from \eqref{mesurenu} that 
\begin{equation*}
\partial_\xi \mathbf{1}_{u(t) > \xi} = -\nu
\end{equation*}
i.e. 
\begin{equation}
\nu = \delta_{u=\xi}.
\end{equation}
Therefore because of \eqref{bornenueps}, we have 
\begin{equation}
\int_{\T^N} |u(t,x)| dx = \int_{\T^N} \int_\R |\xi| d\nu_{t,x}(\xi) dx dt < + \infty
\end{equation}
which concludes the proof.
\end{proof}

\noindent We can actually  strengthen the convergence and get the following result:
\begin{proposition}
$(u_\varepsilon)_{\varepsilon>0}$ converges in $L^p(\Omega \times [0,T] \times \T^N)$ to $u$ for all $p \in [1, + \infty)$.
\end{proposition}

\begin{proof}
Let us recall that we have (up to subsequence)
\begin{equation*}
\lim_{\varepsilon \to 0} \mathbb{E} \int_{[0,T] \times \T^N} h(t,x) \int_\R \phi(\xi) d\nu_{t,x}^\varepsilon dx dt = \mathbb{E} \int_{[0,T] \times \T^N} h(t,x) \int_\R \phi(\xi) d\nu_{t,x} dx dt
\end{equation*}
for all $h \in L^1(\Omega \times [0,T] \times \T^N)$ i.e. 
\begin{equation}
\lim_{\varepsilon \to 0} \mathbb{E} \int_{[0,T] \times \T^N} h(t,x) \phi(u_\varepsilon(t,x)) dx dt = \mathbb{E} \int_{[0,T] \times \T^N} h(t,x) \int_\R \phi(\xi) d\nu_{t,x} dx dt.
\end{equation}
Moreover, we have established previously that $\nu = \delta_{u=\xi}$. Then, we deduce from the compensated compactness theorem that since $\nu$ is a Dirac, $(u_\varepsilon)_{\varepsilon>0}$ converges strongly in $L^p(\Omega \times [0,T] \times \T^N)$ for all $p \in [1, + \infty)$.
\end{proof}

\section{Existence of a solution to the conservation law with stochastic forcing}
\label{conservationlaw}

\subsection{Formal calculus and modified Maxwellian}

\noindent In this section, we will show  in a formal way the passage from the kinetic formulation to the conservation law with stochastic forcing \eqref{stochasticlaw}. Indeed, by Proposition \ref{FegalMu} we know that $F=M_u$ and so satisfies
\begin{equation}
\label{equationF}
\Lambda(x) \partial_\xi F = \mathbf{1}_{u> \xi} - F.
\end{equation}
For any function $b(x,\xi)$, we have
 \begin{equation*}
 \int_\R b(x,\xi) \Lambda(x) \partial_\xi F(x,\xi) d\xi =  \int_{-\infty}^u b(x,\xi) d \xi -  \int_\R b(x,\xi) F(x,\xi) d\xi.
 \end{equation*}
 We denote 
 \begin{equation}
 B(x,v) := \int_{-\infty}^v b(x,\xi) d \xi.
 \end{equation}
 Then we get 
\begin{equation*}
\begin{array}{rcl}
B(x,u) &  = & \dsp{\int_\R b(x,\xi) \Lambda(x) \partial_\xi F(x,\xi) d\xi +   \int_\R b(x,\xi) F(x,\xi) d\xi}\\
 & = & \dsp{\int_\R [b(x,\xi) -  \partial_\xi b(x,\xi) \Lambda(x)] F(x,\xi) d\xi}. 
\end{array}
\end{equation*}
For $b(x,\xi)$ solution of the equation
\begin{equation}
\label{equation_b}
b(x,\xi) - \partial_\xi b(x,\xi) \Lambda(x) = a(x,\xi),
\end{equation}
 we have 
 \begin{equation*}
 B(x,u) = \int_\R a(x, \xi) F(x, \xi) d\xi.
 \end{equation*}
If we impose that $b$ is bounded, by computation of the solution of \eqref{equation_b}, we get
\begin{equation}
B(x,u) = \int_{-\infty}^u b(x,\xi) d \xi= \int_{-\infty}^u \int_0^{+\infty} a(x,\xi+v \Lambda(x)) e^{-v} dv d \xi.
\end{equation}
Quite similarly, for $c(x,\xi)$ solution of the equation 
\begin{equation}
\label{equation_c}
c(x,\xi) - \partial_\xi c(x,\xi) \Lambda(x) = \partial_\xi \Phi(x,\xi),
\end{equation}
we have 
\begin{equation*}
C(x,u) = \int_\R \partial_\xi \Phi(x,\xi) F(x,\xi) d\xi
\end{equation*}
where 
\begin{equation}
C(x,u) :=  \int_{-\infty}^u c(x,\xi) d \xi= \int_{-\infty}^u \int_0^{+\infty} \partial_\xi \Phi(x,\xi+v \Lambda(x)) e^{-v} dv d \xi.
\end{equation}

\noindent Let us start from the kinetic formulation 
\begin{equation}
\label{equationM_u}
dM_u + div_x(a(x,\xi)M_u) dt = \partial_\xi m -  \partial_\xi M_u \Phi dW_t + \frac{1}{2} \partial_\xi(G^2 \partial_\xi M_u) dt. 
\end{equation}
Since because of \eqref{equationF}, we have
\begin{equation*}
\int_\R (M_u - \mathbf{1}_{0>\xi}) d\xi = u,
\end{equation*}
 by integrating \eqref{equationM_u}, we get
\begin{equation*}
du+ div_x B(x,u) dt = C(x,u) dW_t.
\end{equation*} 
In addition, the modified Maxwellian being the solution of the equation \eqref{equationF}, we actually have an explicit expression for it:
\begin{equation}
M_k(x,\xi) = \int_0^{+\infty} \mathbf{1}_{k>\xi}(\xi-\Lambda(x)u)e^{-u} du.
\end{equation}
By  easy computations and Fubini's Theorem, we get the following proposition:
\begin{proposition}
\label{propmaxw}
 We get, for any $k,k' \in \R$,
\begin{enumerate}
\item[(i)] $\dsp{\int_\R (M_k(x,\xi)-\mathbf{1}_{0>\xi}(\xi)) d\xi = k}$,
\item[(ii)] $\dsp{sgn(M_k(x,\xi) - M_{k'}(x,\xi)) = sgn(k-k')}$,
\item[(iii)]  $\dsp{\int_\R |M_k(x,\xi) - M_{k'}(x,\xi)| d\xi = |k-k'|}$,
\item[(iv)] $\begin{array}{l}
\dsp{\int_\R a(x, \xi) M_k(x, \xi) d\xi = B(x,k)}, \\
\dsp{\int_\R \partial_\xi \Phi(x, \xi) M_k(x, \xi) d\xi = C(x,k)}.
\end{array}$
\end{enumerate}
where \begin{equation}
B(x,k) =  \int_{-\infty}^k \int_0^{+\infty} a(x,\xi+v \Lambda(x)) e^{-v} dv d \xi, 
\end{equation}
\begin{equation}
C(x,k) =  \int_{-\infty}^k \int_0^{+\infty} \partial_\xi \Phi(x,\xi+v \Lambda(x)) e^{-v} dv d \xi,
\end{equation}
and 
\begin{equation}
M_k(x,\xi) = \int_0^{+\infty} \mathbf{1}_{k>\xi}(\xi-\Lambda(x)u)e^{-u} du.
\end{equation}
\end{proposition}

\subsection{Existence of a weak solution}

\noindent We are now in position to prove the following result: 

\begin{proposition}
There exists a weak solution $u \in  L^1(\Omega \times [0,T] \times \T^N)$ to the conservation law with stochastic forcing \eqref{stochasticlaw}.
\end{proposition}

\begin{proof}
Let us consider the following test function 
\begin{equation}
\varphi(t,x,\xi) = \Theta_R(\xi) \Psi(t,x)
\end{equation}
with $\Theta_R$ defined in section \ref{stoch_kin_mod}, $\Psi \in \mathcal{C}_c([0,T) \times \T^N)$. Then, applying the kinetic formulation, we get $\mathbb{P}-a.s.$, 
\begin{multline}
\label{form_cin_vers_loi_cons}
\int_0^T <M_u(t) \Theta_R, \partial_t \Psi> dt + <F(0), \varphi(0)>  + \int_0^T <M_u(t) \Theta_R,a. \nabla \Psi> dt \\
=m(\partial_\xi\Theta_R \Psi) - \int_0^T <M_u(t) \partial_\xi \Phi dW_t, \Theta_R \Psi(t)> dt \\
 - \int_0^T <M_u(t)  \Phi dW_t, \partial_\xi \Theta_R \Psi(t)> + \frac{1}{2}  \int_0^T <G^2 \partial_\xi M_u(t), \partial_\xi \Theta_R  \Psi(t)> dt.
\end{multline}
We denote $<.,.>_x$ the duality distribution over $\T_x^N$. Since, we have 
\begin{equation}
0 = <\partial_t \mathbf{1}_{0>\xi}, \varphi> = - <\mathbf{1}_{0>\xi},\partial_t \varphi > + <\mathbf{1}_{0>\xi}, \varphi(0)>,
\end{equation}
using the same arguments as previously, we can rewrite \eqref{form_cin_vers_loi_cons} as follows
\begin{multline}
\int_0^T <\int_\R (M_u(t) - \mathbf{1}_{0>\xi}) \Theta_R d \xi, \partial_t \Psi>_x dt + <\int_\R  (F(0) - \mathbf{1}_{0>\xi})\Theta_R(\xi) d\xi, \Psi(0,x)>_x\\
 + \int_0^T <\int_\R M_u(t) \Theta_R(\xi) a(\xi) d\xi. \nabla \Psi,1>_x dt \\
= m(\partial_\xi \Theta_R  \Psi) - \int_0^T <\int_\R M_u(t) \partial_\xi \Phi \Theta_R(\xi) d\xi  dW_t, \Psi(t)>_x  \\
- \int_0^T <M_u(t)  \Phi dW_t, \partial_\xi \Theta_R \Psi(t)> 
+ \frac{1}{2}  \int_0^T <G^2 \partial_\xi M_u(t), \partial_\xi \Theta_R  \Psi(t)> dt.
\end{multline}
Then, using again the dominated convergence for deterministic and stochastic integrals and the properties of the modified Maxwellian stated in Proposition \ref{propmaxw}, we get when $R \to \infty$, 
\begin{multline}
\int_0^T <u, \partial_t \Psi> dt + <u_0, \Psi(0)>_x + \int_0^T <B(x,u).\nabla \Psi,1> dt  \\
= \int_0^T < - C(x,u) dW_t, \Psi>  
\end{multline}
which concludes the proof.
\end{proof}

\subsection{Krushkov-like entropies}

In this section, though we are not able to obtain the exact Krushkov entropy relations using the deterministic techniques because of the presence of the stochastic term, we still are able to establish some Krushkov-like inequalities. 

\begin{proposition}
For all $\varphi \in \mathcal{C}_c^\infty([0,T) \times \T^N \times \R)$, we have 
\begin{multline}
\label{Krushkovineg}
\int_{[0,T] \times \T^N \times \R} |u(t,x)-k| ~ \partial_t \varphi(t,x,\xi) dt dx d\xi \\ 
+ \int_{[0,T] \times \T^N \times \R}   \left[ \text{sgn}(u(t,x)-k) (B(x,u(t,x)) - B(x,k))  div_x \varphi(t,x,\xi) \right] dt dx d\xi \\
+ \int_{[0,T] \times \T^N \times \R} div_x (B(x,k)) \text{sgn}(u(t,x)-k)  \varphi(t,x,\xi) dt dx d\xi  \\ -  \int_{[0,T] \times \T^N \times \R}  \text{sgn}(u(t,x)-k)  C(x,u(t,x)) dW_t  dx d\xi \leq C_\varphi ~~~~\mathbb{P}\text{-a.s}.
\end{multline} 
where $C_\varphi$ is a constant which depends only on $\varphi$.
\end{proposition}

\begin{proof}
Previously, we have established that  
\begin{equation}
dM_u + div_x(a M_u) dt = \partial_\xi m - \partial_\xi M_u \Phi dW_t + \frac{1}{2}  \partial_\xi (G^2 \partial_\xi M_u) dt 
\end{equation}
in the sense of distributions $\mathbb{P}$-a.s.. We deduce from above that we have 
\begin{equation}
d (M_u - M_k) +  div_x(a (M_u - M_k)) dt = \partial_\xi m - \partial_\xi M_u \Phi dW_t + \frac{1}{2}  \partial_\xi (G^2 \partial_\xi M_u) dt  - div_x (a M_k) dt
\end{equation}
in the sense of distributions $\mathbb{P}$-a.s.\\

\noindent By a similar reasoning, we actually can extend \eqref{formulationcinetique}
to the following test function 
\begin{equation}
\tilde{\varphi}(\omega, t, x, \xi) = \frac{1}{\delta^{N+1}} \eta \pare{\frac{y-x}{\delta}, \frac{\zeta - \xi}{\delta}} P'((M_u -M_k)^\delta(t,y,\zeta)) \Psi(t)
\end{equation}
where $P$ is a $\mathcal{C}^2(\R)$ function such that $P' \in L^\infty(\R)$, $(M_u -M_k)^\delta := (M_u -M_k) * \eta^\delta$ with $\eta$, $\eta^\delta$  defined similarly on $\T^N \times \R$ to $\theta$, $\theta^\delta$ on $[0,T]$ defined on \eqref{theta} and \eqref{thetadelta}. So we obtain that for all $(y,\zeta)$ in $\T^N \times \R$, we have
\begin{multline}
\label{egalpourKru}
d (M_u -M_k)^\delta P'((M_u -M_k)^\delta + div_x (a (M_u -M_k)^\delta)  P'((M_u -M_k)^\delta) dt \\
 = \partial_\xi m * \eta^\delta  P'((M_u -M_k)^\delta) - (\partial_\xi M_u \Phi) * \eta^\delta P'((M_u -M_k)^\delta) dW_t + \frac{1}{2}  \partial_\xi (G^2 \partial_\xi M_u) * \eta^\delta P'((M_u -M_k)^\delta) dt \\
   - div_x (a M_k) * \eta^\delta P'((M_u -M_k)^\delta)  dt - r^\delta P'((M_u -M_k)^\delta) dt
\end{multline} 
where $r^\delta := div_x(a(M_u-M_k)) * \eta^\delta - div_x(a (M_u -M_k)^\delta) $ in the sense of distributions over $(0,T)$ $\mathbb{P}$-a.s.
In order to get the result, we need to consider the above equality against a test function depending on $(y,\zeta)$. Then, though we have the above equality for all $(y,\zeta)$ $\mathbb{P}$-a.s., we need to establish it $\mathbb{P}$-a.s. for all $(y,\zeta)$. It is first easy to deduce from above that the result holds true $\mathbb{P}$-a.s. for all $(y,\zeta) \in (\T \cap \mathbf{Q})^N \times \mathbf{Q}$. Extending this to all $(y,\zeta) \in \T^N \times \mathbf{R}$ is direct by continuity in all the terms except $(\partial_\xi M_u \Phi) * \eta^\delta P'((M_u -M_k)^\delta) dW_t$. Let us now deal with this last term. We consider a test function $\varphi \in \mathcal{C}_c^\infty(\T^N_y \times \R_\zeta)$, and are interested in the following terms:
\begin{equation}
\int_0^T (M_u \partial_\xi g_k)^\delta (t,y,\zeta) \Psi(t)  P'((M_u -M_k)^\delta  (t,y,\zeta)  d\beta_k(t) \varphi (y,\zeta) 
\end{equation} 
and 
\begin{equation}
\int_0^T (M_u g_k) * \partial_\xi \eta^\delta (y,\zeta) \Psi(t)  P'((M_u -M_k)^\delta  (t,y,\zeta))  d\beta_k(t) \varphi (y,\zeta),
\end{equation} 
for $k=1, \dots, d.$
Let us deal with the first one. In order to use a well known result stating that we can find a continuous modification of a stochastic process, we must satisfy in our case the following assumption, for all $0 \leq s \leq t \leq T$, $(y,\zeta)$, $(x, \xi)$ in $\T^N \times \R$,
\begin{multline}
\mathbb{E} \left| \int_0^t (M_u \partial_\xi g_k)^\delta (r,y,\zeta) \Psi(r)  P'((M_u -M_k)^\delta  (r,y,\zeta))  d\beta_k(r) \varphi (y,\zeta) \right. \\
- \left. \int_0^s (M_u \partial_\xi g_k)^\delta (r,x,\xi) \Psi(r)  P'((M_u -M_k)^\delta  (r,x,\xi))  d\beta_k(r) \varphi (x,\xi) \right|^\lambda \\
\leq C (|t-s| + |x-y|)^{N+2+\epsilon}
\end{multline}
with $\lambda >1$, $\epsilon >0$.
Actually, we have the following inequalities for $n >1$:
\begin{equation}
\begin{array}{l}
\dsp{\mathbb{E} \left| \int_0^t (M_u \partial_\xi g_k)^\delta (r,y,\zeta) \Psi(r)  P'((M_u -M_k)^\delta  (r,y,\zeta))  d\beta_k(r) \varphi (y,\zeta) \right.} \\
\dsp{~~~~~~ - \left. \int_0^s (M_u \partial_\xi g_k)^\delta (r,x,\xi) \Psi(r)  P'((M_u -M_k)^\delta  (r,x,\xi))  d\beta_k(r) \varphi (x,\xi) \right|^n}\\
~\\
\dsp{\leq C \mathbb{E} \left| \int_s^t (M_u \partial_\xi g_k)^\delta (r,y,\zeta) \Psi(r)  P'((M_u -M_k)^\delta  (r,y,\zeta))  d\beta_k(r) \varphi (y,\zeta)  \right|^n} \\
 \dsp{~~~~~~ + C \mathbb{E} \left| \int_0^s  \Psi(r)  \left[ (M_u \partial_\xi g_k)^\delta (r,y,\zeta)   P'((M_u -M_k)^\delta  (r,y,\zeta))  \varphi (y,\zeta) \right.  \right.}\\
\dsp{~~~~~~~~~~~~ - \left. \left. (M_u \partial_\xi g_k)^\delta (r,x,\xi))   P'((M_u -M_k)^\delta  (r,x,\xi)) \varphi (x,\xi) \right]  d\beta_k(r)  \right|^n}\\
~\\
\dsp{\leq C \mathbb{E} \left( \int_s^t |(M_u \partial_\xi g_k)^\delta (r,y,\zeta) \Psi(r)  P'((M_u -M_k)^\delta  (r,y,\zeta))  \varphi (y,\zeta) |^2 dr \right)^{n/2}} \\
 \dsp{~~~~~~ + C \mathbb{E} \left( \int_0^s    \left|\Psi(r) \left[ (M_u \partial_\xi g_k)^\delta (r,y,\zeta)   P'((M_u -M_k)^\delta  (r,y,\zeta))  \varphi (y,\zeta) \right.  \right. \right.}\\
\dsp{~~~~~~~~~~~~ - \left. \left. \left. (M_u \partial_\xi g_k)^\delta (r,x,\xi)   P'((M_u -M_k)^\delta  (r,x,\xi)) \varphi (x,\xi) \right] \right|^{2}  dr  \right)^{n/2}}\\
~\\
\dsp{\leq C |t-s|^{n/2} + C \pare{\int_0^s |\varphi (y,\zeta)  - \varphi (x,\xi)|^2 dr}^{n/2}} \\
\leq C \pare{|t-s| + |(y,\zeta) - (x,\xi)|}^n
\end{array}
\end{equation}
by assumptions on $g_k$, $P$ and $\varphi$ and using the fact that $M_u$ is bounded. For $n=N+3$, the assumption is satisfied. The reasoning is the same for the second term. So finally, we can conclude that \eqref{egalpourKru}  holds true $\mathbb{P}$-a.s. for all $(y,\zeta) \in \T^N \times \R$. \\

\noindent The next step consists in using the chain rule. Thus in order to do so, we must reformulate the problem in a Stratonovich version. We have $\mathbb{P}$-a.s. for all $(y,\zeta) \in \T^N \times \R$,
\begin{multline}
\label{egalversKruStrato}
d (M_u -M_k)^\delta P'((M_u -M_k)^\delta) + div_x (a (M_u -M_k)^\delta)  P'((M_u -M_k)^\delta) dt \\
 = \partial_\xi m * \eta^\delta  P'((M_u -M_k)^\delta) - (\partial_\xi M_u \Phi) * \eta^\delta  P'(M_u-M_k) \circ dW_t + \frac{1}{2}  \partial_\xi (G^2 \partial_\xi M_u) * \eta^\delta  P'(M_u-M_k) dt \\
   - div_x (a M_k) * \eta^\delta   P'(M_u-M_k) dt - r^\delta  P'(M_u-M_k) dt + \tau^\delta.
\end{multline} 
in the sense of distributions over $(0,T)$, with $\tau^\delta$ the corrective terms associated to the passage from $(\partial_\xi M_u \Phi) * \eta^\delta   P'(M_u-M_k) dW_t$ to $(\partial_\xi M_u \Phi) * \eta^\delta  P'(M_u-M_k) \circ dW_t$. We first deal with the case $m=0$. We integrate \eqref{egalversKruStrato} against $\varphi(y,\zeta)$. Using in a similar way as previously the dominated convergence theorem for deterministic and stochastic integrals as well as the commutation lemma of Di Perna-Lions \cite{Diperna} to prove the disappearance of $r^\delta$, we obtain 
\begin{multline}
d P(M_u-M_k) + div_x(a P(M_u-M_k)) dt - div_x(a) (P(M_u-M_k)-P'(M_u-M_k)(M_u-M_k))dt \\
 = \partial_\xi M_u \Phi P'(M_u-M_k) \circ dW_t + \frac{1}{2}  \partial_\xi (G^2 \partial_\xi M_u)  P'(M_u-M_k) dt \\
   - div_x (a M_k)  P'(M_u-M_k)  dt + \lim_{\delta \to 0} \tau^\delta.
\end{multline}
in the sense of distributions against test functions of the form $\Psi(t) \varphi(y,\zeta)$ $\mathbb{P}$-a.s. Noticing that the limit of the corrective terms of $(\partial_\xi M_u \Phi) * \eta^\delta  P'(M_u-M_k) \circ dW_t$ is actually the corrective terms of $\partial_\xi M_u \Phi  P'(M_u-M_k) \circ dW_t$, we finally have 
\begin{multline}
d P(M_u-M_k) + div_x(a P(M_u-M_k)) dt - div_x(a) (P(M_u-M_k)-P'(M_u-M_k)(M_u-M_k))dt \\
 = \partial_\xi M_u \Phi P'(M_u-M_k)  dW_t + \frac{1}{2}  \partial_\xi (G^2 \partial_\xi M_u)  P'(M_u-M_k) dt \\
   - div_x (a M_k)  P'(M_u-M_k)  dt
\end{multline}
in the sense of distributions against test functions of the form $\Psi(t) \varphi(y,\zeta)$ $\mathbb{P}$-a.s.
As previously, by density we can extend it to any test function in $\mathcal{C}_c^\infty([0,T) \times \T^N \times \R)$. \\

\noindent We conclude as in the deterministic case by taking $P'=sgn^\alpha$ and $P(\lambda) =  \int_0^\lambda sgn^\alpha(s) ds$ where $sgn^\alpha$ is a smooth regularization of the sign function. By a quite similar reasoning, letting $\alpha$ goes to $0$, we finally get 
\begin{multline}
\label{egalversKruStrato2}
d |M_u -M_k|  + div_x (a |M_u -M_k|)  dt  +  div_x (a M_k) sgn(M_u-M_k) dt \\
 + (\partial_\xi M_u \Phi)  sgn(M_u-M_k) dW_t - \frac{1}{2}  \partial_\xi (G^2 \partial_\xi M_u)   sgn(M_u-M_k) dt \\
   - div_x (a M_k)  sgn(M_u-M_k) dt  = 0
\end{multline}
in the sense of distributions  over $(0,T) \times \T^N \times \R$ $\mathbb{P}$-a.s.\\

\noindent Let us finally go back to the case $m \neq 0$. Then we have the following inequality
\begin{equation}
\begin{array}{l}
\dsp{\int_{\T^N_y \times \R_\zeta} \int_{[0,T] \times \T^N_x \times \R_\xi} \frac{1}{\delta^{N+1}} \partial_\xi \eta(\frac{y-x}{\delta}, \frac{\zeta-\xi}{\delta}) P'(M_u-M_k)(t,y, \zeta) dm(t,x,\xi)  \varphi(t,y,\zeta) dy d\zeta} \\
\dsp{\leq \| P'\|_\infty \int_{[0,T] \times \T^N_x \times \R_\xi}  \partial_\xi \varphi * \eta^\delta(x,\xi)  dm(t,x,\xi)}\\
\dsp{\leq C \| P'\|_\infty m(K)},
\end{array}
\end{equation}
where $K$ is a compact containing the support of $\varphi$. \\

\noindent Finally, integrating the above inequality against a test function of the form $\varphi(t,x,\xi) = \Theta_R(\xi) \psi(t,x)$ and using the properties associated to the modified Maxwellian stated in Proposition \ref{propmaxw} together again with the use of dominated convergence theorem for $R \to + \infty$ leads to the conclusion. 
\end{proof}

\begin{center}
\bf \huge Appendix
\end{center}
\appendix

\section{Well posedness of the stochastic kinetic model}
\label{wellposed}
Our purpose is to establish the existence  of a solution to the stochastic BGK model 
\begin{equation}
\label{BGKsto2}
  \left\{ 
\begin{array}{lcr}
 \dsp{dF^\varepsilon + div_x ( a(x,\xi)F^\varepsilon) dt + \frac{\Lambda(x)}{\varepsilon} \partial_\xi F^\varepsilon dt = \frac{\mathbf{1}_{u_\varepsilon> \xi} - F^\varepsilon}{\varepsilon} dt - \partial_\xi F^\varepsilon \Phi dW + \frac{1}{2} \partial_\xi (G^2 \partial_\xi F^{\varepsilon}) dt,}  \\ 
F_\varepsilon(0)= \mathbf{1}_{u_0> \xi},
\end{array}
\right. 
\end{equation}
where we recall $\dsp{u_\varepsilon(t,x)= \int_\R f_\varepsilon(t,x,\xi) d\xi}$ and $f_\varepsilon = F_\varepsilon + \mathbf{1}_{0> \xi}$. \\

\noindent The proof being quite similar to the one in the case of a classical stochastic BGK model, we will not investigate all the details in the following and we will invite the reader to consult \cite{Hof} for more technical details.\\ 

\noindent In order to use Duhamel's formula, let us first show an interest in the following auxiliary problem
\begin{equation}
  \left\{ 
\begin{array}{lcr}
 \dsp{dX^\varepsilon + div_x ( a(x,\xi)X^\varepsilon) dt + \frac{\Lambda(x)}{\varepsilon} \partial_\xi X^\varepsilon dt = - \partial_\xi X^\varepsilon \Phi dW + \frac{1}{2} \partial_\xi (G^2 \partial_\xi X^{\varepsilon}) dt,}  \\ 
X_\varepsilon(s)= X_0.
\end{array}
\right. 
\end{equation}

\noindent In order to apply the stochastic characteristic of Kunita \cite{Kun}, we introduce a truncated problem. Of course, we will then pass to the limit on the truncation parameter to get our original problem. 

\subsection{A truncated problem}
 
 We are interested in the truncated problem described in section \ref{stoch_kin_mod},
\begin{equation}
  \left\{ 
\begin{array}{lcr}
 \dsp{dX^\varepsilon + div_x ( a^R(x,\xi)X^\varepsilon) dt + \frac{\Lambda(x)}{\varepsilon} \partial_\xi X^\varepsilon dt = - \partial_\xi X^\varepsilon \Phi^R dW + \frac{1}{2} \partial_\xi (G^{R,2} \partial_\xi X^{\varepsilon}) dt,}  \\ 
X^\varepsilon(s)= X_0.
\end{array}
\right. 
\end{equation}
To apply the  method of stochastic characteristics, we must reformulate the problem in Stratonovich form. Using the formula which links It\^o and Stratonovich integrals,  we are able to prove that if $X$ is a $\mathcal{C}^1(\T^N \times \R)$-valued continuous $(\mathcal{F}_t)$-semimartingale whose martingale part is given by $- \int_0^t \partial_\xi \Phi^R dW$ then \\
\begin{equation}
- \int_0^t \partial_\xi \Phi dW + \frac{1}{2} \int_0^t \partial_\xi (G^{R,2} \partial_\xi X) dr = - \int_0^t \partial_\xi X \Phi^R \circ dW + \frac{1}{4} \int_0^t \partial_\xi X \partial_\xi G^{R,2} dr
\end{equation} 
and for $X$ being only a $\mathcal{D}'(\T^N \times \R)$-valued continuous $(\mathcal{F}_t)$-semimartingale, the result is still valid in the sense of distributions (see \cite{Hof} for more details). The truncated problem can then be reformulated as follows
\begin{equation}
\label{TruncStrato}
  \left\{ 
\begin{array}{lcr}
 \dsp{dX^\varepsilon + div_x ( a^R(x,\xi)X^\varepsilon) dt + \frac{\Lambda(x)}{\varepsilon} \partial_\xi X^\varepsilon dt = - \partial_\xi X^\varepsilon \Phi^R \circ dW + \frac{1}{4} \partial_\xi X^\varepsilon \partial_\xi  G^{R,2}  dt,}  \\ 
X^\varepsilon(s)= X_0.
\end{array}
\right. 
\end{equation}
and applying the method, we obtain the following proposition:
\begin{proposition}
Let $R>0$. If $X_0 \in \mathcal{C}^{3,\mu}(\T^N \times \R)$ almost surely, there exists a unique strong solution to \eqref{TruncStrato} that we will denote $X^\varepsilon(t,x,\xi;s)$. \\
In addition, $X^\varepsilon(t,x,\xi;s)$ is a continuous $\mathcal{C}^{3,\nu}$-semimartingale for some $\nu > 0$ and is represented by 
\begin{equation}
\label{expressionX}
X^\varepsilon(t,x,\xi;s) = \exp \pare{\int_s^t - div_x (a^R(\Psi_{\theta,t}^{\varepsilon,R}(x,\xi))) d \theta} X_0(\Psi_{s,t}^{\varepsilon,R}(x,\xi)),
\end{equation}
where $\Psi^R$ is the inverse flow of the stochastic flow associated to the stochastic characteristic system coming from  \eqref{TruncStrato}.
\end{proposition}
\begin{remark}
The stochastic characteristic system associated to \eqref{TruncStrato} is the following
\begin{equation}
\label{CaracTrunc}
  \left\{ 
\begin{array}{lcr}
 \dsp{d \varphi_t^{\varepsilon,R,0} = \frac{\Lambda(\varphi_t^{\varepsilon,R,1}, \dots, \varphi_t^{\varepsilon,R,N})}{\varepsilon} dt - \frac{1}{4} \partial_\xi G^{R,2} (\varphi_t^{\varepsilon,R}) dt + \sum_{k=1}^d g_k^R(\varphi_t^{\varepsilon,R}) \circ d\beta_k(t),}   \\ 
\dsp{d \varphi_t^{\varepsilon,R,i} = a_i^R(\varphi_t^{\varepsilon,R})} dt ~$pour $ ~i=1, \dots, N,\\
d \eta_t^{\varepsilon,R} = \eta_t^{\varepsilon,R} (- div_x(a^R(\varphi_t^{\varepsilon,R}))) dt.
\end{array}
\right. 
\end{equation}
We notice that in our case, it still presents a dependence on $\varepsilon$. 
\end{remark}  

\noindent We denote by $\mathcal{S}^{\varepsilon,R}$ the solution operator of \eqref{TruncStrato}. We have 
\begin{equation}
\mathcal{S}^{\varepsilon,R}(t,s) = \exp \pare{\int_s^t - div_x(a^R(\Psi_{\theta,t}^{\varepsilon,R}(x,\xi)))) d\theta} X_0(\Psi_{s,t}^{\varepsilon,R}(x,\xi)).
 \end{equation}
The domain of definition of the solution operator can be extended to $X_0$ only defined almost everywhere since diffeomorphisms preserve sets of measure zero. Nevertheless, the resulting process will no longer be a strong solution. We obtain the following properties for the operator $\mathcal{S}^{\varepsilon,R}$. 
\begin{proposition}
\label{PropOp}
Let $R>0$, $\varepsilon>0$ and $\mathcal{S}^{\varepsilon,R} = \{ \mathcal{S}^{\varepsilon,R}(t,s), 0 \leq s \leq t \leq T \}$  be defined as previously. Then 
\begin{enumerate}
\item[(i)] $\mathcal{S}^{\varepsilon,R}$ is a family of bounded linear operators on $L^1(\Omega \times \T^N \times \R)$ with a unit operator norm, meaning for any $X_0 \in L^1(\Omega \times \T^N \times \R)$, $0 \leq s \leq t \leq T $,
\begin{equation}
\| \mathcal{S}^{\varepsilon,R}(t,s) X_0\|_{L^1(\Omega \times \T^N \times \R)} \leq \| X_0\|_{L^1(\Omega \times \T^N \times \R)}.
\end{equation}
\item[(ii)]  $\mathcal{S}^{\varepsilon,R}$ verifies the semi-group law 
\begin{equation}
\begin{array}{ll}
\mathcal{S}^{\varepsilon,R}(t,s) = \mathcal{S}^{\varepsilon,R}(t,r) \circ \mathcal{S}^{\varepsilon,R}(r,s), & 0 \leq s \leq r \leq  t \leq T,\\
\mathcal{S}^{\varepsilon,R}(s,s)= Id, & 0 \leq s  \leq T.
\end{array}
\end{equation}
\end{enumerate}
\end{proposition}

\begin{proof}[Sketch of the proof]
The proof being essentially the same as the one in \cite{Hof}, let us only recall here the main ideas and insist on what is a bit different in our case. 
The idea is the following: we first work with a regularized initial data $X_0^\delta$ defined as follows $X_0^\delta(\omega)=(X_0(\omega) \ast h_\delta) k_\delta$ where $(k_\delta)$ is a smooth truncation on $\R$ i.e. $k_\delta(\xi)=k(\delta \xi)$ in order to build a unique strong solution $X^{\varepsilon,\delta}=\mathcal{S}^{\varepsilon,R}(t,s) X_0^\delta$ to the following system
\begin{equation}
\label{eqregul}
  \left\{ 
\begin{array}{lcr}
 \dsp{dX^{\varepsilon,\delta}+ div_x ( a^R(x,\xi)X^{\varepsilon,\delta}) dt + \frac{\Lambda(x)}{\varepsilon} \partial_\xi X^{\varepsilon,\delta} dt = - \partial_\xi X^{\varepsilon,\delta} \Phi^R dW + \frac{1}{2} \partial_\xi (G^{R,2}  \partial_\xi X^{\varepsilon,\delta})  dt,}  \\ 
X^{\varepsilon,\delta}(s)= X_0^\delta.
\end{array}
\right. 
\end{equation}
The aim is the following: we integrate the equation \eqref{eqregul} with respect to the variables $\omega,x,\xi$ and get 
\begin{multline}
 \mathbb{E} \int_{\T^N} \int_\R X^{\varepsilon,\delta}(t,x,\xi) d\xi dx ~ +~   \mathbb{E} \int_s^t \int_{\T^N} \int_\R  div_x ( a^R(x,\xi)X^{\varepsilon,\delta}(r,x, \xi)) d\xi dx dr \\ ~+~  \mathbb{E}  \int_s^t \int_{\T^N} \int_\R \frac{\Lambda(x)}{\varepsilon} \partial_\xi X^{\varepsilon,\delta}(r,x, \xi)  d\xi dx dr\\ 
   = \mathbb{E} \int_{\T^N} \int_\R X_0^\delta(x,\xi) d\xi dx - \mathbb{E} \int_s^t \int_{\T^N} \int_\R \partial_\xi X^{\varepsilon,\delta}(r,x, \xi) \Phi^R dW d\xi dx dr\\
   + \mathbb{E} \int_s^t \int_{\T^N} \int_\R \frac{1}{2}  \partial_\xi (G^{R,2}(x,\xi)  \partial_\xi X^{\varepsilon,\delta}(r,x, \xi) ) d\xi dx dr.
\end{multline}
Then, we want to show that all the terms disappears except for the first term of the left-hand side and the first term of the right-hand side. Let us first focus on the stochastic integral. All we need to do is: 
\begin{enumerate}
\item[(1)] Prove that this integral is a well defined martingale with zero expected value.
\item[(2)] Use the stochastic Fubini theorem to interchange integrals with respect to $x,\xi$ and the stochastic one. 
\end{enumerate}
Then we will have proven the disappearance of the second term of the right-hand side. \\ 

\noindent In order to do $(1)$, we must prove that $\mathbb{E} \int_s^t | \partial_\xi X^{\varepsilon,\delta}(r,x, \xi) g_k^{R,2}| dr < + \infty$. The expression of $X^{\varepsilon,\delta}$ being slightly more complicated in our case, we must claim that not only $\partial_\xi \Psi_{s,r}^{\varepsilon,R}(x,\xi)$ but also $\partial_{x_i} \Psi_{s,r}^{\varepsilon,R}$ and $\partial_{\xi} \partial_{x_i} \Psi_{s,r}^{\varepsilon,R}$ for $i=1, \dots, N$ solves a backward stochastic differential equations with bounded coefficients (see \cite[Theorem 4.6.5 and Corollary 4.6.6]{Kun}) since $g_k^R$, $X_0^\delta$ and all the partial derivatives of all order of $a_i^R$ for $i=1, \dots, N$ are bounded. Thus, those solutions possess moments of any order which are bounded in $0 \leq s \leq t \leq T$, $x \in \T^N$, $\xi \in \R$. By assumption \eqref{positivitediv}, we have $div_x(a(x,\xi)) \geq 0$ for all $x, \xi$ which leads to the statement (1). \\

\noindent To prove (2), we must establish that 
\begin{equation}
\dsp{\int_{\T^N} \int_\R \pare{\mathbb{E} \int_s^T ( |\partial_\xi X^{\varepsilon, \delta} g_k^R(x, \xi)|^2) dr}^{1/2} d\xi dx   < + \infty}.
\end{equation}
 By the expression \eqref{expressionX} of $X^{\varepsilon, \delta}$ , we know that $\partial_\xi X^{\varepsilon, \delta}$ only contain terms mentionned previously which all are bounded or possess moments of any order which are bounded. Then it is enough to prove that $X_0^\varepsilon(\Psi_{s,r}^{\varepsilon,R}(x,\xi))$ and $\nabla_{x,\xi} X_0^\varepsilon(\Psi_{s,r}^{\varepsilon,R}(x,\xi))$ are compactly supported in $\xi$ for all $s,r,x,\mathbb{P}-$ a.s. thanks to some growth control on the stochastic flow to conclude (for more details see \cite{Hof}). \\

\noindent Finally, the second term of the left-hand side disappears thanks to periodic boundary conditions and the third terms of the left-hand side and the right-hand side disappears because of the support compact in $\xi$ of respectfully $X^{\varepsilon,\delta}$ and $G^{R,2}$. Therefore at the end, we obtain
\begin{equation}
 \mathbb{E} \int_{\T^N} \int_\R X^{\varepsilon,\delta}(t,x,\xi) d\xi dx =  \mathbb{E} \int_{\T^N} \int_\R X_0^\delta(x,\xi) d\xi dx .
\end{equation}
The statement $(i)$ of Proposition \ref{PropOp} is then obtained by passing to the limit in $\delta$ and a use of the Fatou lemma. The statement $(ii)$ of Proposition \ref{PropOp} is straightforward using the flow property of $\Psi$. 
\end{proof}

\noindent Moreover, a use of the regularization $X_0^\delta$ together with the dominated convergence theorem and the dominated convergence theorem for stochastic integrals directly leads to the following result:

\begin{corollary}
\label{cont_tps}
Let $R > 0$. If $X_0$ is a $\mathcal{F}_s \otimes \mathcal{B}(\T^N) \otimes  \mathcal{B}(\R)$- measurable initial data belonging to $L^\infty(\Omega \times \T^N \times \R)$, then there exists a weak solution $X^\varepsilon \in L_{\mathcal{P}_s}^\infty(\Omega \times [s,T] \times \T^N \times \R)$ to \eqref{TruncStrato} meaning for any $\phi \in \mathcal{C}_c^\infty( \T^N \times \R)$, a.e. $t \in [s,T]$, $\mathbb{P}$ a.s., 
\begin{multline}
<X^\varepsilon(t), \phi> = <X^\varepsilon(0), \phi> + \int_s^t <X^\varepsilon(r), a^R. \nabla \phi> dr + \frac{1}{\varepsilon}  \int_s^t <X^\varepsilon(r), \Lambda(x) \partial_\xi \phi> dr \\
+ \sum_{k=1}^d    \int_s^t <X^\varepsilon(r),  \partial_\xi(g_k^R \phi) > d \beta_k(r) + \frac{1}{2} \int_s^t <X^\varepsilon(r),  \partial_\xi(G^{R,2} \partial_\xi \phi) > dr.
\end{multline} 
Moreover, $X^\varepsilon$ is represented by $X^\varepsilon = \mathcal{S}^{\varepsilon,R}(t,s) X_0$. In particular, $t \mapsto < \mathcal{S}^{\varepsilon,R}(t,s) X_0, \phi> $ is a continuous $(\mathcal{F}_t)_{t \geq s}$- semimartingale. 
\end{corollary}

\begin{remark}
We point out that in our situation, with the presence of the term $div_x(a^R(x,\xi))$ which is not necessary equal to zero, the use of the commutation lemma of Di Perna-Lions (\cite[Lemma II.1]{Diperna}) proves that the uniqueness of a weak solution no longer holds precisely in the case where $div_x(a^R(x,\xi))$ is not equal to zero. 
\end{remark}

\subsection{Passage to the non truncated problem} 

Our aim is to derive the existence of a weak solution to 
\begin{equation}
\label{Eqnontronque}
  \left\{ 
\begin{array}{lcr}
 \dsp{dX^{\varepsilon}+ div_x ( a(x,\xi)X^{\varepsilon}) dt + \frac{\Lambda(x)}{\varepsilon} \partial_\xi X^{\varepsilon} dt = - \partial_\xi X^{\varepsilon} \Phi dW + \frac{1}{2} \partial_\xi (G^{2}  \partial_\xi X^{\varepsilon})  dt,}  \\ 
X^{\varepsilon}(s)= X_0.
\end{array}
\right. 
\end{equation}
\begin{proposition}
\label{existencenontrunc}
If  $X_0 $ is a $\mathcal{F}_s \otimes \mathcal{B}(\T^N) \otimes  \mathcal{B}(\R)$- measurable initial data belonging to $L^\infty(\Omega \times \T^N \times \R)$, then there exists a weak solution $X^\varepsilon \in L_{\mathcal{P}_s}^\infty(\Omega \times [s,T] \times \T^N \times \R)$ to \eqref{Eqnontronque}. Moreover, it is represented by 
\begin{equation}
\mathcal{S}^\varepsilon(t,s)(\omega,x,\xi) := \lim_{R \to + \infty} [\mathcal{S}^{\varepsilon,R}(t,s) X_0](\omega,x,\xi), ~~~~~ 0 \leq s\leq t \leq T.
\end{equation}
\end{proposition}

\begin{proof}
We denote $\tau^R(s,x,\xi) = \inf \{ t \geq s, |\varphi_{s,t}^{\varepsilon,R,0}(x,\xi)| > \frac{R}{2}\}$ with the convention $\inf \varnothing = T$. $\tau^R(s,x,\xi)$ is a stopping time with respect to the filtration $(\mathcal{F}_t)_{t \geq s}$ for any $s \in [0,t]$, $x \in \T^N$, $\xi \in \R$. For more clarity, we do not write the dependence on $s,x,\xi$ of $\tau^R$ in the following. \\

\noindent Let us first prove that $(\tau^R)_{R \in \mathbf{N}}$ is a non decreasing sequence such that $\dsp{\tau_\infty := \lim_{R \to \infty} \tau^R}$ is equal to $T$ almost surely. We know that for any $R >0$, the process $\varphi^{\varepsilon,R,0}$ satisfies the equation 
\begin{equation}
d\varphi_t^{\varepsilon,R,0} = \frac{\Lambda(\varphi_t^{\varepsilon,R,1}, \dots, \varphi_t^{\varepsilon,R,N})}{\varepsilon} dt + \sum_{k=1}^d g_k^R(\varphi_t^{\varepsilon,R}) d \beta_k(t).
\end{equation}
We denote $\tilde{\tau}_R := \inf \{ \tau^R, \tau^{R+1}\}$. The equations satisfy by $\varphi_t^{\varepsilon,R,0}$ and $\varphi_t^{\varepsilon,R+1,0}$ being exactly the same on $[0, \tilde{\tau}^R]$, then the coefficients being at least Lipschitz we know that $\varphi_{s,t}^{\varepsilon,R,0}(x,\xi)= \varphi_{s,t}^{\varepsilon,R+1,0}(x,\xi)$ for $t \in [0,\tilde{\tau}^R]$. So, we get that $\tau^{R} \leq \tau^{R+1}$. Indeed, let us reason by absurd and suppose that $\tau^R > \tau^{R+1}$. Then for $t \in [\tau^{R+1}, \tau^{R}]$, we have 
\begin{equation*}
|\varphi_{s,t}^{\varepsilon,R,0}(x,\xi)| = |\varphi_{s,t}^{\varepsilon,R+1,0}(x,\xi)| > \frac{R+1}{2} > \frac{R}{2}
\end{equation*}
which is a contradiction with the fact that $\tau^R$ is the infimum of the times realizing this condition. So the sequence $(\tau^R)_{R \in \mathbf{N}}$ is a non decreasing and we can define the almost sure limit
\begin{equation}
\tau_\infty := \lim_{R \to \infty} \tau^R.
\end{equation}
We know that $\tau_\infty$ is still a stopping time. We denote by $\varphi_{s,t}^{\varepsilon,R,0}(x,\xi) = \varphi_{s,t}^{\varepsilon,R,0}(x,\xi)$ for $t \in [0, \tau^R]$. The process is defined on $[0, \tau_\infty]$ and it is a solution of 
\begin{equation}
d\varphi_t^{\varepsilon,R,0} = \frac{\Lambda(\varphi_t^{\varepsilon,R,1}, \dots, \varphi_t^{\varepsilon,R,N})}{\varepsilon} dt + \sum_{k=1}^d g_k(\varphi_t^{\varepsilon,R}) d \beta_k(t).
\end{equation}

\noindent Let us now establish some useful estimates to conclude: 
\begin{equation*}
\begin{array}{rcl}
\varphi_{s,t}^{\varepsilon,R,0} & = & \dsp{\varphi_{in}^{\varepsilon,R,0} + \int_s^t 
\frac{\Lambda(\varphi_r^{\varepsilon,R,1}, \dots, \varphi_r^{\varepsilon,R,N})}{\varepsilon} dr + \int_s^t  \sum_{k=1}^d g_k^R(\varphi_r^{\varepsilon,R}) d \beta_k(r)}\\
|\varphi_{s,t}^{\varepsilon,R,0}|^2  &\leq &  \dsp{C \pare{|\varphi_{in}^{\varepsilon,R,0}|^2 + \pare{\int_s^t 
\va{\frac{\Lambda(\varphi_r^{\varepsilon,R,1}, \dots, \varphi_r^{\varepsilon,R,N})}{\varepsilon}} dr}^2 + \va{\int_s^t  \sum_{k=1}^d g_k^R(\varphi_r^{\varepsilon,R}) d \beta_k(r)}^2 }}.
\end{array}
\end{equation*}
Thus by taking the expectation, using well-known results on martingales and assumption \eqref{inegG}, we get 
\begin{equation*}
\begin{array}{rcl}
\mathbb{E} |\varphi_{s,t}^{\varepsilon,R,0}|^2 & = & \dsp{C~ \pare{ \mathbb{E} |\varphi_{in}^{\varepsilon,R,0}|^2 + \frac{M^2 T^2}{\varepsilon^2} + \mathbb{E}  \int_s^t  |\sum_{k=1}^d g_k^R(\varphi_r^{\varepsilon,R})|^2 dr }}\\
 &\leq &  \dsp{C~  \pare{\mathbb{E} |\varphi_{in}^{\varepsilon,R,0}|^2 + \frac{M^2 T^2}{\varepsilon^2} + \mathbb{E}  \int_s^t  (1+ |\varphi_{s,t}^{\varepsilon,R,0}|^2) dr} }\\
  &\leq &  \dsp{C~  \pare{\mathbb{E} |\varphi_{in}^{\varepsilon,R,0}|^2 + \frac{M^2 T^2}{\varepsilon^2} + T + \mathbb{E}  \int_s^t  |\varphi_{s,t}^{\varepsilon,R,0}|^2 dr }}.
\end{array}
\end{equation*}
Indeed, $\Lambda$ being continuous on $\T^N$ compact, there exists a constant $M$ such that for all $x \in \T^N$, $|\Lambda(x)| \leq M$. By Gronwall lemma, we obtain
\begin{equation}
\mathbb{E} |\varphi_{s,t}^{\varepsilon,R,0}|^2  \leq  \dsp{C~  (\mathbb{E} |\varphi_{in}^{\varepsilon,R,0}|^2 + \frac{M^2 T^2}{\varepsilon^2} + T) \exp{(CT)} } =: \mathcal{C}^1(\varphi_{in}^{\varepsilon,R,0},T,\varepsilon).
\end{equation} 
Moreover, we also have the following 
\begin{equation}
\begin{array}{rcl}
\dsp{\sup_{t \in [s,T]} \va{\varphi_{s,t}^{\varepsilon,R,0}}^2 }& \leq & \dsp{C \left( |\varphi_{in}^{\varepsilon,R,0}|^2 + \sup_{t \in [s,T]}  \va{\int_s^t \frac{\Lambda(\varphi_r^{\varepsilon,R,1}, \dots, \varphi_r^{\varepsilon,R,N})}{\varepsilon} dr}^2 \right.} \\
& & \dsp{~~~~~~~~~~~~~~~~~~~~~~~~~~~~~~~~~~~~~  \left. + \sup_{t \in [s,T]} \va{ \int_s^t  \sum_{k=1}^d g_k^R(\varphi_r^{\varepsilon,R}) d\beta_k(r)}^2 \right)  }. 
\end{array}
\end{equation}
By the same arguments as previously, we deduce that 
\begin{equation}
\begin{array}{rcl}
\dsp{\mathbb{E} \sup_{t \in [s,T]} |\varphi_{s,t}^{\varepsilon,R,0}|^2 }& \leq & \dsp{C \pare{\mathbb{E} |\varphi_{in}^{\varepsilon,R,0}|^2  +  \frac{M^2 T^2}{\varepsilon^2} +  \mathbb{E} \sup_{t \in [s,T]} \va{ \int_s^t  \sum_{k=1}^d g_k^R(\varphi_r^{\varepsilon,R}) d\beta_k(r)}^2  }} \\
 & \leq & \dsp{C \pare{\mathbb{E} |\varphi_{in}^{\varepsilon,R,0}|^2  +  \frac{M^2 T^2}{\varepsilon^2} + 4 \int_s^T \mathbb{E} |\sum_{k=1}^d g_k^R(\varphi_r^{\varepsilon,R})|^2 dr }}\\
  & \leq & \dsp{C \pare{\mathbb{E} |\varphi_{in}^{\varepsilon,R,0}|^2  +  \frac{M^2 T^2}{\varepsilon^2} +  \int_s^T \mathbb{E}  (1+ |\varphi_{s,t}^{\varepsilon,R,0}|^2) dr }}\\
  & \leq & \dsp{C \pare{\mathbb{E} |\varphi_{in}^{\varepsilon,R,0}|^2  +  \frac{M^2 T^2}{\varepsilon^2} + T +  \int_s^T \mathbb{E}  |\varphi_{s,t}^{\varepsilon,R,0}|^2 dr }}\\
  & \leq & \dsp{C \pare{\mathbb{E} |\varphi_{in}^{\varepsilon,R,0}|^2  +  \frac{M^2 T^2}{\varepsilon^2} + T +  \int_0^T \mathcal{C}^1(\varphi_{in}^{\varepsilon,R,0},T,\varepsilon) dr }=: \mathcal{C}^2(\varphi_{in}^{\varepsilon,R,0},T, \varepsilon}). 
\end{array}
\end{equation}
Finally, applying the Markov inequality, we get 
\begin{equation}
\begin{array}{rcl}
\dsp{\mathbb{P} \pare{\sup_{t \in [s,T]} |\varphi_{s,t}^{\varepsilon,R,0}| \geq \frac{R}{2}}}& \leq & \dsp{\frac{4}{R^2} \mathbb{E} \pare{\dsp{\sup_{t \in [s,T]} |\varphi_{s,t}^{\varepsilon,R,0}|^2}}}\\
 & \leq & \dsp{\frac{4~\mathcal{C}^2(\varphi_{in}^{\varepsilon,R,0},T, \varepsilon)}{R^2}}
\end{array}
\end{equation}
and then 
\begin{equation}
\dsp{\mathbb{P} \pare{\sup_{t \in [s,T]} |\varphi_{s,t}^{\varepsilon,R,0} | < \frac{R}{2}} \geq 1 -  \frac{4~\mathcal{C}^2(\varphi_{in}^{\varepsilon,R,0},T, \varepsilon)}{R^2}}.
 \end{equation}
 $(\tau^R)_{R \in \mathbf{N}}$ being a non decreasing sequence, we finally obtain that $\mathbb{P}(\tau_\infty = T) =1$ passing to the limit in $R$.\\
 
\noindent Still because of uniqueness we have $\mathcal{S}^{\varepsilon,R+1}(t,s) X_0 = \mathcal{S}^{\varepsilon,R}(t,s) X_0 $ on $[0, \tau^R(s,x,\xi)]$. So, we define
 \begin{equation}
 \mathcal{S}^\varepsilon(t,s) X_0(\omega,x,\xi) := \lim_{R \to + \infty} [\mathcal{S}^{\varepsilon,R}(t,s) X_0](\omega,x,\xi), ~~~~~ 0 \leq s\leq t \leq T.
 \end{equation}
 which is a solution to \eqref{Eqnontronque}. 
\end{proof}

\begin{corollary} \label{propoperateur} We have the following properties:
\begin{enumerate}
\item[(i)]  $\mathcal{S}^\varepsilon = \{\mathcal{S}^\varepsilon(t,s), 0 \leq s \leq t \leq T\}$ is a family of bounded operators on $L^1(\Omega \times \T^N \times \R)$  with a unit operator norm meaning for any $X_0 \in L^1(\Omega \times \T^N \times \R)$,  $ 0 \leq s \leq t \leq T$, 
\begin{equation}
\dsp{\| \mathcal{S}^\varepsilon(t,s) X_0 \|_{L^1(\Omega \times \T^N \times \R)} \leq \|  X_0 \|_{L^1(\Omega \times \T^N \times \R)}}.
\end{equation}
\item[(ii)]  $\mathcal{S}^{\varepsilon,R}$ verifies the semi-group law 
\begin{equation}
\begin{array}{ll}
\mathcal{S}^{\varepsilon}(t,s) = \mathcal{S}^{\varepsilon}(t,r) \circ \mathcal{S}^{\varepsilon}(r,s), & 0 \leq s \leq r \leq  t \leq T, \\
\mathcal{S}^{\varepsilon}(s,s)= Id, & 0 \leq s  \leq T.
\end{array}
\end{equation}
\end{enumerate}
\end{corollary}
\noindent The proof is straightforward using the definition of $\mathcal{S}^{\varepsilon}$ and Proposition \ref{PropOp}.

\noindent We can conlude to the existence of a solution for the stochastic kinetic equation. 

\begin{proposition}
For any $\varepsilon >0$, there exists a weak solution of the stochastic BGK equation with a high field scaling denoted by $F_\varepsilon$. Moreover, $F_\varepsilon$ is represented by 
\begin{equation}
\label{represent_int}
\dsp{F_\varepsilon(t) = e^{-t/ \varepsilon} \mathcal{S}^\varepsilon(t,0) \mathbf{1}_{u_0> \xi} + \frac{1}{\varepsilon} \int_0^t e^{-\frac{t-s}{\varepsilon}} \mathcal{S}^\varepsilon(t,s)\mathbf{1}_{u_\varepsilon(s) > \xi}ds}.
\end{equation}
\end{proposition}

\begin{proof}[Sketch of the proof]
Noticing that $F_\varepsilon$ is not integrable with respect to $\xi$, we introduce the process $h_\varepsilon(t) := F_\varepsilon - \mathcal{S}^\varepsilon(t,0) \mathbf{1}_{0 > \xi}$. Then, by Proposition \ref{existencenontrunc}, we get that $h_\varepsilon$ solves the following system
\begin{equation}
\label{BGKh}
  \left\{ 
\begin{array}{l}
 \dsp{d h^\varepsilon + div_x ( a(x,\xi) h^\varepsilon) dt + \frac{\Lambda(x)}{\varepsilon} \partial_\xi h^\varepsilon dt = \frac{(\mathbf{1}_{u_\varepsilon> \xi} - \mathcal{S}^\varepsilon(t,0) \mathbf{1}_{0 > \xi}) - h^\varepsilon}{\varepsilon} dt - \partial_\xi h^\varepsilon \Phi dW_t}\\
 \dsp{~~~~~~~~~~~~~~~~~~~~~~~~~~~~~~~~~~~~~~~~~~~~~~~~~~~~~~~~~~~~~~~~~~~~~~~~~~~~~~~~~~~~~~~~~~~~~~~+ \frac{1}{2} \partial_\xi (G^2 \partial_\xi h^{\varepsilon}) dt,}  \\ 
h_\varepsilon(0)= \chi_{u_0}.
\end{array}
\right. 
\end{equation}
The proof is then quite classical. Indeed, using Duhamel's formula, the problem \eqref{BGKh} can be rewritten as follows
\begin{equation}
h_\varepsilon(t)=e^{-t/\varepsilon} \mathcal{S}^\varepsilon(t,0) \chi_{u_0} + \frac{1}{\varepsilon} \int_0^t e^{-\frac{t-s}{\varepsilon}} \mathcal{S}^\varepsilon(t,s)\croch{\mathbf{1}_{u_\varepsilon(s)> \xi} - \mathcal{S}^\varepsilon(s,0) \mathbf{1}_{0 > \xi}}ds
\end{equation}
and it is then reduced to a fixed point method. We define the mapping $\mathcal{L}$: 
\begin{equation}
(\mathcal{L}g)(t)=e^{-t/s}  \mathcal{S}^\varepsilon(t,0) \chi_{u_0} + \frac{1}{\varepsilon} \int_0^t e^{-\frac{t-s}{\varepsilon}} \mathcal{S}^\varepsilon(t,s)\croch{\mathbf{1}_{v(s)> \xi} - \mathcal{S}^\varepsilon(s,0) \mathbf{1}_{0 > \xi}}ds
\end{equation}
where $v(s)=\int_\R (g(s,\xi) + \mathcal{S}^\varepsilon(s,0) \mathbf{1}_{0 > \xi} - \mathbf{1}_{0 > \xi} )d\xi$. We show that the mapping $\mathcal{L}$ is a contraction on $L^\infty(0,T;L^1(\Omega \times \T^N))$ using Corollary \ref{propoperateur} and assumptions on the initial data (see \cite{Hof} for more details).  Then we can conclude to the existence of a unique fixed point $h_\varepsilon$ in $L^\infty(0,T;L^1(\Omega \times \T^N))$. Moreover, using the properties of the solution operator of Corollary \ref{propoperateur} and using the fact that $\chi_{u_0} = \mathbf{1}_{u_0> \xi} -  \mathbf{1}_{0 > \xi}$, we finally obtain
\begin{equation}
\begin{array}{rcl}
F_\varepsilon(t) & = &\dsp{ h_\varepsilon(t) + \mathcal{S}^\varepsilon(t,0) \mathbf{1}_{0 > \xi}}\\
& = & \dsp{e^{-t/ \varepsilon} \mathcal{S}^\varepsilon(t,0) \mathbf{1}_{u_0> \xi} + \frac{1}{\varepsilon} \int_0^t e^{-\frac{t-s}{\varepsilon}} \mathcal{S}^\varepsilon(t,s)\mathbf{1}_{u_\varepsilon(s) > \xi}ds}
\end{array}
\end{equation}
which concludes the proof. 
\end{proof}

\begin{remark}
As a consequence of Corollary \ref{cont_tps}, $F_\varepsilon(t)$ satisfies the following : $t \mapsto <F_\varepsilon(t), \phi>$ is a continuous $(\mathcal{F}_t)$- semimartingale for any $\phi \in \mathcal{C}_c^\infty(\T^N \times \R)$.
\end{remark}

 ~~~\\

\noindent \textbf{Acknowledgment.} The author thanks F. Berthelin and F. Delarue for many helpful discussions. 

\bibliographystyle{plain}
\bibliography{Biblio}

\noindent Nathalie Ayi\\
Laboratory J.A. Dieudonn\'e\\
UMR CNRS-UNS $N^0$ $7351$\\
University of Nice Sophia-Antipolis\\
Parc Valrose\\
06108 Nice Cedex 2\\
France\\

\noindent Member of the COFFEE team (INRIA)\\
\noindent e-mail: nathalie.ayi@unice.fr

\end{document}